\documentclass[english,american,aps,preprint]{revtex4-1}
\usepackage[T1]{fontenc}
\usepackage[latin9]{inputenc}
\usepackage{color}
\usepackage{babel}
\usepackage{verbatim}
\usepackage{amsthm}
\usepackage{amsmath}
\usepackage{amssymb}
\usepackage{graphicx}
\usepackage{esint}
\usepackage[unicode=true,pdfusetitle,
 bookmarks=true,bookmarksnumbered=false,bookmarksopen=false,
 breaklinks=false,pdfborder={0 0 1},backref=false,colorlinks=false]
 {hyperref}
\usepackage{breakurl}

\makeatletter

\@ifundefined{textcolor}{}
{%
 \definecolor{BLACK}{gray}{0}
 \definecolor{WHITE}{gray}{1}
 \definecolor{RED}{rgb}{1,0,0}
 \definecolor{GREEN}{rgb}{0,1,0}
 \definecolor{BLUE}{rgb}{0,0,1}
 \definecolor{CYAN}{cmyk}{1,0,0,0}
 \definecolor{MAGENTA}{cmyk}{0,1,0,0}
 \definecolor{YELLOW}{cmyk}{0,0,1,0}
}
  \theoremstyle{definition}
  \newtheorem{defn}{\protect\definitionname}
  \theoremstyle{remark}
  \newtheorem{rem}{\protect\remarkname}
  \theoremstyle{plain}
  \newtheorem{prop}{\protect\propositionname}
\theoremstyle{plain}
\newtheorem{thm}{\protect\theoremname}
  \theoremstyle{plain}
  \newtheorem{lem}{\protect\lemmaname}
  \theoremstyle{definition}
  \newtheorem*{example*}{\protect\examplename}

\usepackage[labelfont=bf, justification=raggedright,format=plain]{caption}
\usepackage[usenames,dvipsnames]{xcolor}
\hypersetup{colorlinks=true, linkcolor=MidnightBlue, citecolor=Brown}

\@ifundefined{showcaptionsetup}{}{%
 \PassOptionsToPackage{caption=false}{subfig}}
\usepackage{subfig}
\makeatother

  \addto\captionsamerican{\renewcommand{\definitionname}{Definition}}
  \addto\captionsamerican{\renewcommand{\examplename}{Example}}
  \addto\captionsamerican{\renewcommand{\lemmaname}{Lemma}}
  \addto\captionsamerican{\renewcommand{\propositionname}{Proposition}}
  \addto\captionsamerican{\renewcommand{\remarkname}{Remark}}
  \addto\captionsenglish{\renewcommand{\definitionname}{Definition}}
  \addto\captionsenglish{\renewcommand{\examplename}{Example}}
  \addto\captionsenglish{\renewcommand{\lemmaname}{Lemma}}
  \addto\captionsenglish{\renewcommand{\propositionname}{Proposition}}
  \addto\captionsenglish{\renewcommand{\remarkname}{Remark}}
  \providecommand{\definitionname}{Definition}
  \providecommand{\examplename}{Example}
  \providecommand{\lemmaname}{Lemma}
  \providecommand{\propositionname}{Proposition}
  \providecommand{\remarkname}{Remark}
\addto\captionsamerican{\renewcommand{\theoremname}{Theorem}}
\addto\captionsenglish{\renewcommand{\theoremname}{Theorem}}
\providecommand{\theoremname}{Theorem}

\begin{document}

\title{An Autonomous Dynamical System Captures all LCSs in Three-Dimensional
Unsteady Flows }

\author{David Oettinger}

\affiliation{Institute of Mechanical Systems, ETH Zürich, Leonhardstrasse 21,
8092 Zürich, Switzerland}

\author{George Haller}

\email{georgehaller@ethz.ch}

\affiliation{Institute of Mechanical Systems, ETH Zürich, Leonhardstrasse 21,
8092 Zürich, Switzerland}
\begin{abstract}
Lagrangian coherent structures (LCSs) are material surfaces that shape
finite-time tracer patterns in flows with arbitrary time dependence.
Depending on their deformation properties, elliptic and hyperbolic
LCSs have been identified from different variational principles, solving
different equations. Here we observe that, in three dimensions, initial
positions of all variational LCSs are invariant manifolds of the same
autonomous dynamical system, generated by the intermediate eigenvector
field, $\xi_{2}(x_{0})$, of the Cauchy-Green strain tensor. This
$\xi_{2}$-system allows for the detection of LCSs in any unsteady
flow by classic methods, such as Poincaré maps, developed for autonomous
dynamical systems. As examples, we consider both steady and time-aperiodic
flows, and use their dual $\xi_{2}$-system to uncover both hyperbolic
and elliptic LCSs from a single computation. 
\end{abstract}
\maketitle

\textbf{Tracer patterns, such as the funnel of a tornado, suggest
the emergence of coherence even in complex unsteady flows. As a mathematical
tool for analyzing the dynamics behind time-evolving tracer patterns,
Lagrangian coherent structures (LCSs) represent a generalization of
classic invariant manifolds to non-autonomous systems. In three dimensions,
the available LCS types (hyperbolic and elliptic) have been identified
from different principles. Here we observe that for any unsteady flow
in three dimensions, there is a single autonomous dynamical system
capturing all LCSs. Specifically, this dynamical system is given
by the intermediate eigenvector field of the Cauchy-Green strain tensor.
Our observation enables the identification of LCSs in any unsteady
flow by standard numerical methods for autonomous systems.}

\section{Introduction\label{sec:Introduction}}

Lagrangian coherent structures (LCSs, \citep{Haller2015}) are exceptional
surfaces of trajectories that shape tracer patterns in unsteady flows
over finite time intervals of interest. By their sustained coherence,
LCSs are observed as barriers to transport. In autonomous or time-periodic
dynamical systems, classic codimension-one invariant manifolds play
a similar role (e.g., Komolgorov-Arnold-Moser (KAM) tori \citep{Arnold1989specific}).
In the time-aperiodic and finite-time setting, this role is taken
over by LCSs as codimension-one invariant manifolds (\emph{material
surfaces}) in the \emph{extended} phase space.

Material surfaces are abundant, yet most impose no observable coherence.
LCSs are distinguished material surfaces that have exceptional impact
on nearby material surfaces.  Since various distinct mechanisms producing
such impact are known \citep{Haller2015}, no unique mathematical
approach has been available to locate all the LCSs in a given flow.
Instead, separate mathematical methods and computational algorithms
exist for the three main LCS types: hyperbolic LCSs as generalizations
of stable and unstable manifolds \citep{Haller2011a,Blazevski2014};
elliptic LCSs as generalizations of invariant tori \citep{Haller2013,Blazevski2014,Oettinger2016};
and, in two dimensions, parabolic LCSs as generalized jet cores \citep{Farazmand2014a}. 

Several works \citep{Haller2011a,Haller2013,Farazmand2014a,Blazevski2014,Oettinger2016}
have implemented properties that distinguish LCSs from generic material
surfaces by requiring the LCSs to yield a critical value for a relevant
quantity of material deformation. The criticality requirement defining,
for instance, repelling hyperbolic LCSs (generalized stable manifolds)
is that these material surfaces exert locally strongest repulsion
\citep{Blazevski2014}. Elliptic LCSs in two dimensions, on the other
hand, can be obtained as stationary curves of an averaged stretching
functional \citep{Haller2013}. For the remaining LCS types in two
and three dimensions, similar variational theories are available \citep{Haller2011a,Farazmand2014a,Blazevski2014,Oettinger2016}. 

All the variational LCS theories \citep{Haller2011a,Haller2013,Farazmand2014a,Blazevski2014,Oettinger2016}
provide particular direction fields to which initial LCS positions
must be either tangent (in two dimensions) or normal (in three dimensions).
Later LCS positions can then be constructed by forward or backward
advection under the flow map. 

In two dimensions, LCSs are simply material curves \citep{Haller2015,Haller2013,Farazmand2014a}.
Initial LCS positions can thus be identified by computing integral
curves of (time-independent) direction fields defined in the two-dimensional
phase space. Obtaining initial-time LCS surfaces in three dimensions
\citep{Blazevski2014,Oettinger2016}, on the other hand, is significantly
more complicated: One has to construct entire surfaces perpendicular
to a given three-dimensional direction field. The presently available
approach to extracting these surfaces is to sample the flow domain
using two-dimensional reference planes, and then, within each plane,
integrate direction fields that are perpendicular to the imposed LCS
normal field. This procedure typically yields a high number of integral
curves, which are candidates for intersection curves between unknown
LCSs and the respective slice of the flow domain. As a second step,
from this large collection of candidate curves, one has to identify
smaller families of curves that can be interpolated into surfaces.
Moreover, since the normal fields depend on the type of LCS, one has
to repeat this complicated analysis for each LCS type \textcolor{black}{\citep{Blazevski2014,Oettinger2016}}.

Here we observe that initial positions of all available variational
LCSs in three dimensions share a common tangent vector field: the
intermediate eigenvector field, $\xi_{2}(x_{0})$, of the right Cauchy-Green
strain tensor. This allows us to seek all LCSs in three dimensions
as invariant manifolds of the autonomous dynamical system generated
by the $\xi_{2}$-field. The evolution of the $\xi_{2}$-system takes
place in the initial configuration of the underlying non-autonomous
system, but contains averaged information about the non-autonomous
flow. The autonomous $\xi_{2}$-system is hence dual to the original
unsteady flow. Equivalently, LCS final positions are invariant manifolds
of the intermediate eigenvector field, $\eta_{2}(x_{1})$, of the
left Cauchy-Green strain tensor. 

Instead of identifying LCSs in three dimensions from various two-dimensional
direction fields \textcolor{black}{\citep{Blazevski2014,Oettinger2016}},
we therefore need to consider only a single three-dimensional direction
field. We then locate LCSs by familiar numerical methods developed
for autonomous dynamical systems.

\texttt{\textcolor{black}{}}

\section{Set-up for Lagrangian coherent structures in 3D\label{sec:setup}}

Here we briefly review the mathematical foundations for Lagrangian
coherent structures in three dimensions \citep{Haller2015}. We consider
ordinary differential equations of the form
\begin{equation}
\dot{x}=u(x,t),\quad x\in U,\quad t\in I,\label{eq:flowdef}
\end{equation}
where $U$ is a domain in the Euclidean space $\mathbb{R}^{3}$; $I$
is a time interval; $u$ is a smooth mapping from the extended phase
space $U\times I$ to $\mathbb{R}^{3}$.  The setting in (\ref{eq:flowdef})
includes time-aperiodic, non-autonomous dynamical systems for which
asymptotic limits are undefined. 

We consider a finite time interval $[t_{0},t_{1}]\subset I$ and
denote a trajectory of (\ref{eq:flowdef}) passing through a point
$x_{0}$ at time $t_{0}$ by $x(t;t_{0},x_{0})$. \textcolor{black}{For
points $x_{0}$ where the trajectory $x(t;t_{0},x_{0})$ is defined
for all times $t\in[t_{0},t_{1}]$, }we introduce the flow map $F_{t_{0}}^{t}(x_{0}):=x(t;\, t_{0},\, x_{0}).$
Denoting the support of $F_{t_{0}}^{t}$ by $D$, the flow map is
a diffeomorphism onto its image $F_{t_{0}}^{t}(D)$.   Hence the
inverse $\left(F_{t_{0}}^{t}\right)^{-1}$ exists, and, in particular,
$\left(F_{t_{0}}^{t}\right)^{-1}=F_{t}^{t_{0}}.$ 
\begin{defn}[\textbf{Material surface}]
Consider a set of initial positions forming a two-dimensional surface
$\text{\ensuremath{\mathcal{M}}}(t_{0})$ at time $t_{0}$ in $U$.
Its time-$t$ image, $\text{\ensuremath{\mathcal{M}}}(t)$, is obtained
under the flow map as 
\begin{equation}
\text{\ensuremath{\mathcal{M}}}(t)=\text{\ensuremath{F{}_{t_{0}}^{t}(\mathcal{M}}}(t_{0})).\label{eq:material_surface}
\end{equation}
The union of all time-$t$ images, $\cup_{t\in[t_{0},t_{1}]}\text{\ensuremath{\mathcal{M}}}(t)$,
is a hypersurface in the extended phase space $U\times I$, called
a\emph{ material surface}. Unless we consider a specific time-$t^{*}$
image $\text{\ensuremath{\mathcal{M}}}(t^{*})$ by fixing time to
a certain value $t^{*}\in[t_{0},t_{1}]$, we refer to the entire material
surface simply by the notation $\text{\ensuremath{\mathcal{M}}}(t)$.

\end{defn}
Any material surface is an invariant manifold in the extended phase
space $U\times I$ and, hence, cannot be crossed by integral curves
$(x(t;t_{0},x_{0}),t)$. Only special material surfaces, however,
create coherence in the phase space $U$ and hence act as observable
transport barriers. Such material surfaces are generally called \emph{Lagrangian
coherent structures} (LCSs). 

Quantifying material coherence in a general non-autonomous system
requires considering (\ref{eq:flowdef}) for a fixed time interval
$[t_{0},t_{1}]$. This reflects the observation that coherent structures
in truly unsteady flows are generally transient. (See also \citep{Haller2015}.)
Accordingly, any LCS is defined with respect to the fixed time interval
$[t_{0},t_{1}]$.\emph{ }(Thus, in applications where multiple time
intervals $[t_{0},t_{1}]$ are relevant, the LCSs need to be determined
separately for each time interval.) 

Viewed in the phase space $U$, LCSs are time-dependent surfaces,
even if the underlying dynamical system (\ref{eq:flowdef}) is autonomous.
LCS positions at different times are related via (\ref{eq:material_surface}).

In applications, even if the flow map $F_{t_{0}}^{t}$ is available
for all $t\in[t_{0},t_{1}]$, it remains challenging to detect and
parametrize all the a priori unknown LCSs. This, fortunately, need
not be done in the extended phase space: Since the flow map applied
to any LCS position $\text{\ensuremath{\mathcal{M}}}(t^{*})$ uniquely
generates any required time-$t$ image $\text{\ensuremath{\mathcal{M}}}(t)$,
we can fix the time \textcolor{black}{$t^{*}$ to an arbitrary value
in $[t_{0},t_{1}]$ and parametrize }$\text{\ensuremath{\mathcal{M}}}(t^{*})$\textcolor{black}{{}
in the phase space $U$. For simplicity, we generally choose $t^{*}=t_{0}$.}
(For attracting hyperbolic LCSs, however, it is advantageous to parametrize
$\text{\ensuremath{\mathcal{M}}}(t_{1})$ instead of $\text{\ensuremath{\mathcal{M}}}(t_{0})$,
see Sec. \ref{sub:ABCaperiodic5}.). The difficulty remains in that
almost any conceivable surface from the domain $D$ evolves incoherently
under the flow, and hence does not define an LCS $\text{\ensuremath{\mathcal{M}}}(t)$
(cf. Fig. 1). We therefore need additional properties that, for any
time-aperiodic flow, distinguish LCSs from generic material surfaces\@.
\begin{figure}[h]
\centering{}\includegraphics[width=0.8\columnwidth]{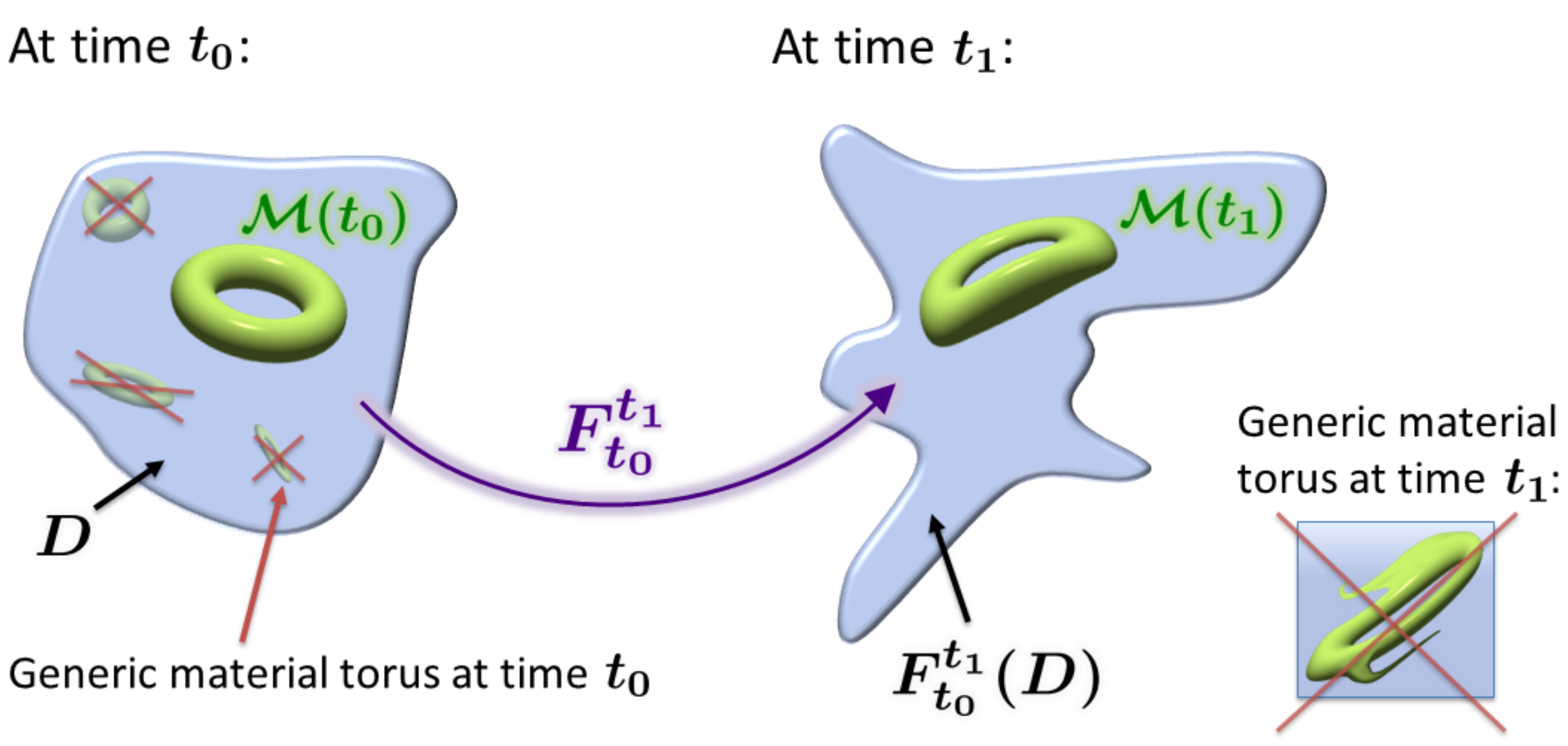}\protect\caption{Schematic of an elliptic LCS $\text{\ensuremath{\mathcal{M}}}(t)$,
obtained as a toroidal surface $\text{\ensuremath{\mathcal{M}}}(t_{0})$
in the flow domain $D$ at time $t_{0}$. Up to rotations and translations,
the time-$t_{1}$ image, $\text{\ensuremath{\mathcal{M}}}(t_{1})$,
is only moderately deformed relative to $\text{\ensuremath{\mathcal{M}}}(t_{0})$
and does not display additional features, such as filaments. (In the
context of fluid dynamics, such an LCS could capture a coherently
evolving vortex ring in a three-dimensional unsteady flow.) Generic
tori in $D$, on the other hand, are expected to evolve incoherently
under the flow $F_{t_{0}}^{t_{1}}$ and thus do not yield LCSs.}
\label{fig:LCSdetection}
\end{figure}

\section{Review of variational approaches to Lagrangian coherent structures
in 3D\label{sec:Review3d} }

Within the general class of three-dimensional flows with arbitrary
time dependence (\ref{eq:flowdef}), several types of material surfaces
can be viewed as coherently evolving. Each of them defines a distinct
type of LCS. Three LCS types have so far been identified: hyperbolic
repelling and attracting LCSs (generalized stable and unstable manifolds)
\citep{Blazevski2014}, and elliptic LCSs (generalized invariant tori
or invariant tubes) \citep{Blazevski2014,Oettinger2016}. 

Hyperbolic LCSs are locally most repelling or attracting material
surfaces \citep{Blazevski2014}. To express this property mathematically,
we introduce the normal repulsion $\rho$ of a material surface $\text{\ensuremath{\mathcal{M}}}(t)$
between times $t_{0}$ and $t_{1}$ (cf. Fig. \ref{fig:RepulsionShear}).
\begin{figure}[h]
\includegraphics[width=0.6\columnwidth]{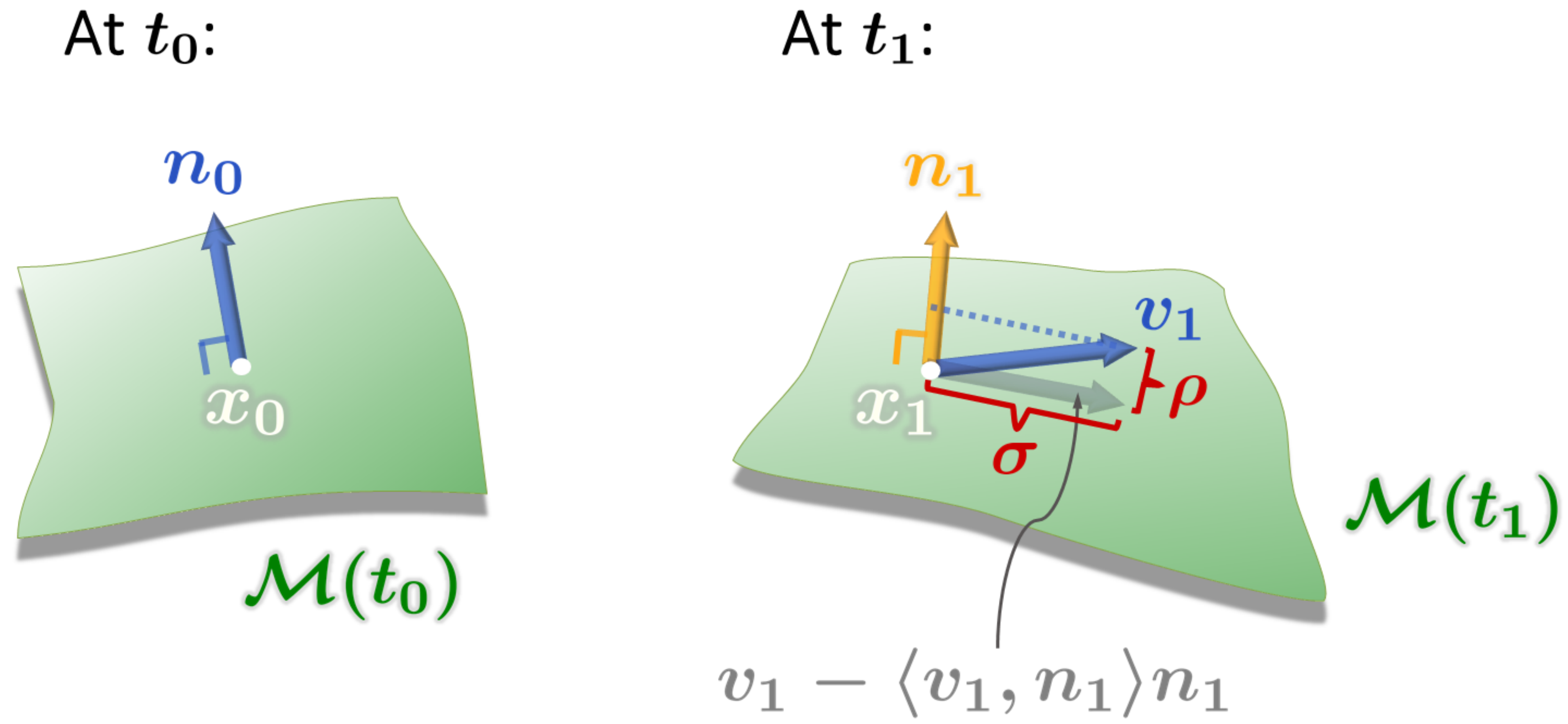}\protect\caption{Definitions of normal repulsion $\rho$, cf. (\ref{eq:repulsion}),
and the tangential shear $\sigma$, cf. (\ref{eq:shear}).}
\label{fig:RepulsionShear}
\end{figure}
 Specifically, at an arbitrary point $x_{0}$ in $\text{\ensuremath{\mathcal{M}}}(t_{0})$,
we consider a unit surface normal $n_{0}(x_{0})$: Mapping $n_{0}(x_{0})$
under the linearized flow $DF_{t_{0}}^{t_{1}}(x_{0})$ from $t_{0}$
to $t_{1}$ yields a vector $v_{1}(x_{1})=DF_{t_{0}}^{t_{1}}(x_{0})n_{0}(x_{0})$,
where $x_{1}=F_{t_{0}}^{t_{1}}(x_{0})$ is a point in $\text{\ensuremath{\mathcal{M}}}(t_{1})$.
The vector $v_{1}(x_{1})$ will generally neither be of unit length
nor perpendicular to the surface $\text{\ensuremath{\mathcal{M}}}(t_{1})$.
Denoting the unit normal of $\text{\ensuremath{\mathcal{M}}}(t_{1})$
at $x_{1}$ by $n_{1}(x_{1})$, we introduce the normal repulsion
$\rho$ as
\begin{equation}
\rho=||\left\langle v_{1},n_{1}\right\rangle n_{1}||=\left\langle v_{1},n_{1}\right\rangle ,\label{eq:repulsion}
\end{equation}
where $\left\langle .,.\right\rangle $ is the Euclidean scalar product,
and $||.||$ is the Euclidean norm. A large value of $\rho$ means
that the component of $v_{1}(x_{1})$ normal to the surface $\text{\ensuremath{\mathcal{M}}}(t_{1})$
is large and, thus, material elements that were initially aligned
with $n_{0}(x_{0})$ appear repelled from $\text{\ensuremath{\mathcal{M}}}(t_{1})$.
Similarly, if the normal component of $v_{1}(x_{1})$ is small, then
the components of $v_{1}(x_{1})$ tangent to $\text{\ensuremath{\mathcal{M}}}(t_{1})$
must be large, corresponding to attraction of material elements aligned
with $n_{0}(x_{0})$ to the surface $\text{\ensuremath{\mathcal{M}}}(t_{1})$.
Formally, we consider the normal repulsion as a function of the initial
position $x_{0}$ and the surface normal $n_{0}(x_{0})$, i.e., $\rho=\rho(x_{0},n_{0})$.
With this convention, $\text{\ensuremath{\mathcal{M}}}(t_{0})$ determines
$\rho$. We now use $\rho$ to define hyperbolic LCSs as most repelling
or attracting material surfaces:
\begin{defn}[\textbf{Repelling and attracting hyperbolic LCS \citep{Blazevski2014}}]
\label{def:hyperbolicLCS} A smooth material surface $\mathcal{M}(t)$
is a \emph{repelling (or attracting) hyperbolic LCS} if the unit normals
$n_{0}(.)$ of $\mathcal{M}(t_{0})$ maximize (or minimize) the normal
repulsion function $\rho$ among all perturbations $n_{0}(.)\mapsto\tilde{n}_{0}(.)$,
with $\tilde{n}_{0}:\mathcal{M}(t_{0})\rightarrow S^{2}$ denoting
an arbitrary unit vector field. 

We additionally require $\rho>1$ ($\rho<1$) for repelling (attracting)
hyperbolic LCSs, which is automatically satisfied for incompressible
flows.
\end{defn}

Motivated by KAM tori and coherent vortex rings in fluid flows, we
require elliptic LCSs to be tubular surfaces in the phase space. By
a tubular surface, we mean a smooth surface that is diffeomorphic
to a torus, cylinder, sphere or paraboloid. In order to capture the
most influential tubular surfaces, Fig. \ref{fig:RepulsionShear}
suggests considering elliptic LCSs as surfaces maximizing the tangential
shear $\sigma$ under perturbations to the surface normal \citep{Blazevski2014}.
This Lagrangian shear $\sigma$ is defined as
\begin{equation}
\sigma=||v_{1}-\left\langle v_{1},n_{1}\right\rangle n_{1}||=||v_{1}-\rho\, n_{1}||\label{eq:shear}
\end{equation}
(cf. Fig. \ref{fig:RepulsionShear}). We consider the tangential shear
$\sigma$ as a function of the initial position $x_{0}$ and the surface
normal $n_{0}(x_{0})$, i.e., we write $\sigma=\sigma(x_{0},n_{0})$.
\begin{defn}[\textbf{Shear-maximizing elliptic LCS \citep{Blazevski2014}}]
\label{def:shearLCS} A tubular material surface $\mathcal{M}(t)$
is an \emph{elliptic LCS} if the unit normals $n_{0}(.)$ of $\mathcal{M}(t_{0})$
maximize the tangential shear function $\sigma$ among all perturbations
$n_{0}(.)\mapsto\tilde{n}_{0}(.)$, with $\tilde{n}_{0}:\mathcal{M}(t_{0})\rightarrow S^{2}$
denoting an arbitrary unit vector field. 
\end{defn}
As pointed out in \citep{Oettinger2016}, due to ever-present numerical
inaccuracies, it is difficult to construct entire tubular surfaces
that satisfy the strict requirement of pointwise maximal shear. A
less restrictive definition of elliptic LCSs has been obtained recently
by considering material surfaces $\mathcal{M}(t)$ that stretch nearly
uniformly under the flow \citep{Oettinger2016}. Considering any point
$x_{0}$ in $\text{\ensuremath{\mathcal{M}}}(t_{0})$, the linearized
flow $DF_{t_{0}}^{t_{1}}$ maps any vector $e_{0}(x_{0})$ from the
tangent space $\text{\ensuremath{T_{x_{0}}\mathcal{M}}}(t_{0})$ to
a vector $e_{1}(x_{1})$ in $T_{x_{1}}\text{\ensuremath{\mathcal{M}}}(t_{1})$,
where $x_{1}=F_{t_{0}}^{t_{1}}(x_{0})$. We define $\mathcal{M}(t)$
as \emph{nearly uniformly stretching} at $x_{0}$ if all tangent vectors
$e_{0}(x_{0})$ satisfy 
\begin{equation}
||e_{1}(x_{1})||=\lambda(x_{0})\cdot||e_{0}(x_{0})||\text{\quad with\quad\ }\lambda(x_{0})\in[\sigma_{2}(x_{0})\cdot(1-\Delta),\sigma_{2}(x_{0})\cdot(1+\Delta)],\label{eq:nearly uniformly}
\end{equation}
where $\sigma_{2}(x_{0})$ is the intermediate singular value of $DF_{t_{0}}^{t_{1}}(x_{0})$
(introduced below, cf. (\ref{eq:df-on-xi})); and $\Delta$ is a small
stretching deviation ($0\leq\Delta\ll1$). As shown in \citep{Oettinger2016},
setting $\lambda(x_{0})=\sigma_{2}(x_{0})$ (i.e., $\Delta=0$) is
the only way to obtain a material surface that is exactly uniformly
stretching at $x_{0}$ (cf. Fig. \ref{fig:LambdaStretching}). 
\begin{figure}[h]
\includegraphics[width=0.6\textwidth]{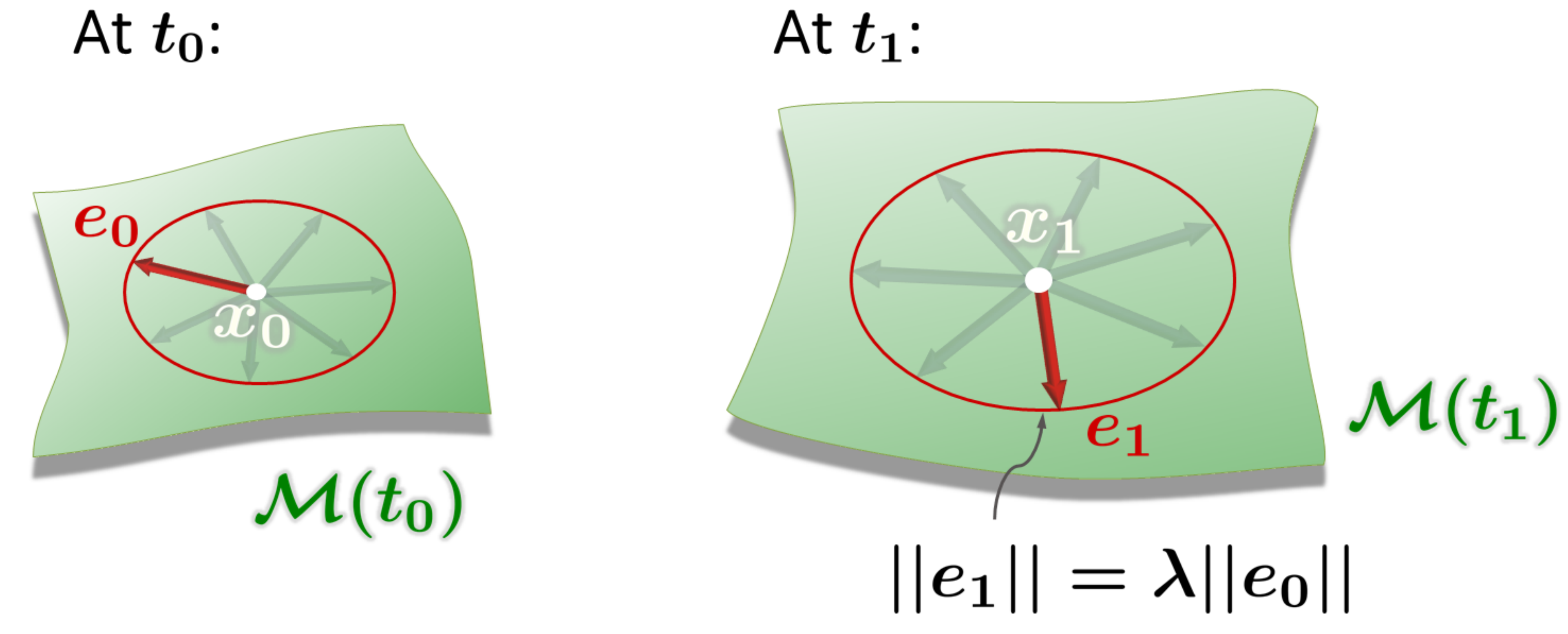} \protect\caption{Local deformation of a pointwise uniformly stretching surface (cf.
(\ref{eq:nearly uniformly})): All tangent vectors based at $x_{0}$
stretch exactly by the same factor of $\lambda(x_{0})$ between times
$t_{0}$ and $t_{1}$.}

\label{fig:LambdaStretching}
\end{figure}
 
\begin{defn}[\textbf{Near-uniformly stretching elliptic LCS \citep{Oettinger2016}}]
\label{def:ellipticLCS}A tubular material surface $\mathcal{M}(t)$
is an \emph{elliptic LCS} if it is \emph{nearly uniformly stretching}
at any point in $\text{\ensuremath{\mathcal{M}}}(t_{0})$. \end{defn}
\begin{rem}
In \citep{Oettinger2016}, the stretching deviation $\Delta$ is chosen
to be constant on $\text{\ensuremath{\mathcal{M}}}(t_{0})$. We could,
however, let $\Delta$ vary on $\text{\ensuremath{\mathcal{M}}}(t_{0})$
and still obtain valid elliptic LCSs (as long as $0\leq\Delta\ll1$).
Requiring exact uniform stretching ($\Delta=0$) would be similarly
restrictive as requiring maximal tangential shear (cf. Definition
\ref{def:shearLCS}). 
\end{rem}

\begin{rem}
Since $\sigma_{2}(x_{0})$ is given by the problem and generally not
a constant function, the factor $\lambda=\lambda(x_{0})$ varies within
the surface $\text{\ensuremath{\mathcal{M}}}(t_{0})$ even when $\Delta=0$.
In two dimensions, however, it is possible to construct elliptic LCSs
that stretch by a factor $\lambda$ that is constant on $\text{\ensuremath{\mathcal{M}}}(t_{0})$
\citep{Haller2013}. 
\end{rem}

\begin{rem}
Other types of distinguished material surfaces revealing elliptic
LCSs are level sets of the polar rotation angle \citep{farazmand2015polar}
and level sets of the Lagrangian-averaged vorticity \citep{haller2015defining}.
These approaches are based on the notion of rotational coherence rather
than stretching, and are hence not directly related to the variational
approaches we review here.
\end{rem}

From the linearization of the flow map $F_{t_{0}}^{t_{1}}$, we can
derive explicit geometric conditions for both hyperbolic and elliptic
LCSs (Definitions \ref{def:hyperbolicLCS}--\ref{def:ellipticLCS}).
These conditions are expressible in terms of eigenvectors and eigenvalues
of the left and right Cauchy-Green strain tensors (cf. Remark \ref{remark:Cauchy-Green}
below). A fully equivalent, yet simpler picture is provided by the
singular-value decomposition (SVD) of the the linearized flow map
$DF_{t_{0}}^{t_{1}}(x_{0})$: The linearized flow map $DF_{t_{0}}^{t_{1}}(x_{0})$
(also called deformation gradient) maps vectors from the tangent space
at $x_{0}$ onto their time-$t_{1}$ images in the tangent space
at the point $x_{1}=F_{t_{0}}^{t_{1}}(x_{0})$. (Since the flow domain
$U$ is in the Euclidean space $\mathbb{R}^{3}$, each of these tangent
spaces is simply $\mathbb{R}^{3}$ as well.)  In particular, $DF_{t_{0}}^{t_{1}}(x_{0})$
maps its three right-singular vectors $\xi_{1,2,3}(x_{0})$ onto its
three left-singular vectors $\eta_{1,2,3}(x_{1})$, i.e., 
\begin{equation}
DF_{t_{0}}^{t_{1}}(x_{0})\xi_{i}(x_{0})=\sigma_{i}(x_{0})\cdot\eta_{i}(x_{1}),\qquad i=1,2,3,\label{eq:df-on-xi}
\end{equation}
see Fig. \ref{fig:svd} and \citep{TrefethenBauSpecific}. 
\begin{figure}[h]
\centering{}\includegraphics[width=0.6\columnwidth]{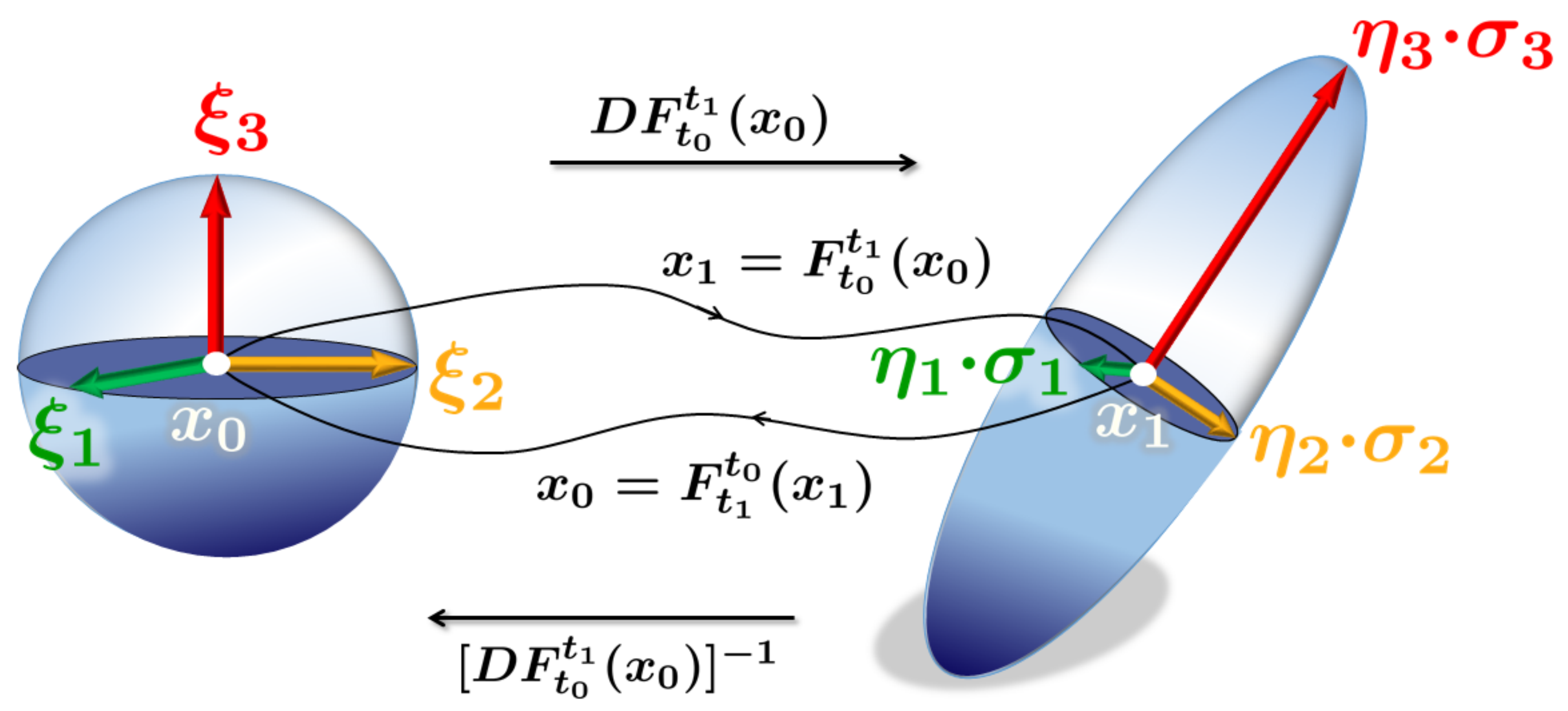}\protect\caption{The deformation gradient $DF_{t_{0}}^{t_{1}}$ mapping its right-singular
vectors $\xi_{1,2,3}$ onto its left-singular vectors $\eta_{1,2,3}$.}
\label{fig:svd}
\end{figure}
 The singular vectors $\xi_{1,2,3}(x_{0})$ and the $\eta_{1,2,3}(x_{1})$
are unit vectors. Both the $\xi_{1,2,3}(x_{0})$ and the $\eta_{1,2,3}(x_{1})$
define an orthonormal basis of $\mathbb{R}^{3}$. The stretch factors
$\sigma_{1,2,3}(x_{0})$ in (\ref{eq:df-on-xi}) are the singular
values of $DF_{t_{0}}^{t_{1}}(x_{0})$, which we assume to be distinct
and ordered so that 
\begin{equation}
0<\sigma_{1}(x_{0})<\sigma_{2}(x_{0})<\sigma_{3}(x_{0}).\label{eq:svals_ascending}
\end{equation}
The available LCS definitions \citep{Blazevski2014,Oettinger2016}
do not consider points where two singular values are equal. 

We illustrate the kinematic role of the right-singular vectors $\xi_{1,2,3}(x_{0})$
by considering the stretch factor of a vector $v(x_{0})$, defined
as 
\begin{equation}
\Lambda_{t_{0}}^{t_{1}}(x_{0},v(x_{0}))=\frac{\left\Vert DF_{t_{0}}^{t_{1}}(x_{0})\, v(x_{0})\right\Vert }{||v(x_{0})||}.\label{eq:stretch-ratio}
\end{equation}

\emph{ }Since\emph{ $\sigma_{1}(x_{0})<\sigma_{2}(x_{0})<\sigma_{3}(x_{0})$,}
any vector $v(x_{0})$ parallel to $\xi_{3}(x_{0})$ maximizes the
stretch factor $\Lambda_{t_{0}}^{t_{1}}(x_{0},.)$ among all vectors
from $\mathbb{R}^{3}$. The direction $\xi_{1}(x_{0})$, on the other
hand, minimizes $\Lambda_{t_{0}}^{t_{1}}(x_{0},.)$. We thus refer
to the (right-) singular vector $\xi_{2}(x_{0})$ as the \emph{intermediate}
(right-) singular vector of $DF_{t_{0}}^{t_{1}}(x_{0})$. In many
applications, the flow $F_{t_{0}}^{t_{1}}$ is volume-preserving (incompressible).
Incompressibility means that $\sigma_{1}\sigma_{2}\sigma_{3}=1$ holds
everywhere. Together with $0<\sigma_{1}<\sigma_{2}<\sigma_{3}$, this
implies that $\sigma_{2}$ is the singular value closest to unity
(cf. Appendix \ref{sec:sigma2_unity}). Accordingly, $\xi_{2}$ is
the singular vector closest to unit stretching (i.e., $\Lambda_{t_{0}}^{t_{1}}=1$).

The backward-time flow map $F_{t_{1}}^{t_{0}}$ yields a similar interpretation
for the left-singular vectors $\eta_{1,2,3}(x_{1})$: The backward-time
deformation gradient, $DF_{t_{1}}^{t_{0}}(x_{1})$, satisfies $DF_{t_{1}}^{t_{0}}(x_{1})=\left[DF_{t_{0}}^{t_{1}}(x_{0})\right]^{-1}$.
The right-singular vectors of $DF_{t_{1}}^{t_{0}}(x_{1})$ are, therefore,
precisely the vectors $\eta_{1,2,3}(x_{1})$; the left-singular vectors
of $DF_{t_{1}}^{t_{0}}(x_{1})$ are the $\xi_{1,2,3}(x_{0})$.  In
backward time, the $\eta_{1,2,3}(x_{1})$ hence play a similar role
to $\xi_{1,2,3}(x_{0})$ in forward time. With the singular values
of $DF_{t_{1}}^{t_{0}}(x_{1})$ being $[\sigma_{1,2,3}(x_{0})]^{-1}$,
it is, however, the vector $\eta_{1}(x_{1})$ that maximizes $\Lambda_{t_{1}}^{t_{0}}.$
This means, the direction of largest stretching in backward time is
$\eta_{1}(x_{1})$. Similarly, the vector $\eta_{3}(x_{1})$ coincides
with the direction of least stretching in backward time; and $\eta_{2}(x_{1})$
is the intermediate (right-) singular vector of $DF_{t_{1}}^{t_{0}}(x_{1})$. 
\begin{rem}
\label{remark:Cauchy-Green}By introducing the right Cauchy-Green
strain tensor 
\begin{equation}
C_{t_{0}}^{t_{1}}(x_{0})=\left[DF_{t_{0}}^{t_{1}}(x_{0})\right]^{T}DF_{t_{0}}^{t_{1}}(x_{0})\,,
\end{equation}
 where the $T$-superscript indicates transposition, we recover the
singular vectors $\xi_{1,2,3}(x_{0})$ as eigenvectors of $C_{t_{0}}^{t_{1}}(x_{0})$.
The associated eigenvalues of $C_{t_{0}}^{t_{1}}(x_{0})$ are $\lambda_{1,2,3}(x_{0})=[\sigma_{1,2,3}(x_{0})]^{2}.$
Similarly, introducing the left Cauchy-Green strain tensor \citep{marsdenhughes1994specific}
as 
\begin{equation}
B_{t_{0}}^{t_{1}}(x_{1})=DF_{t_{0}}^{t_{1}}(x_{0})\left[DF_{t_{0}}^{t_{1}}(x_{0})\right]^{T},
\end{equation}
 where $x_{0}=F_{t_{1}}^{t_{0}}(x_{1})$, the left-singular vectors
$\eta_{1,2,3}(x_{1})$ are the eigenvectors of $B_{t_{0}}^{t_{1}}(x_{1})$.
The use of $C_{t_{0}}^{t_{1}}$ and $B_{t_{0}}^{t_{1}}$ is a common
approach in the LCS literature \citep{Haller2015,HallerSapsis2011}.
As it is, however, numerically advantageous to use SVD instead of
eigendecomposition \citep{Watkins2005NoEIGButSVD,karrasch2015attracting},
we will not use the Cauchy-Green strain tensors here. 
\end{rem}

From the above it follows that the hyperbolic LCSs introduced in Definition
\ref{def:hyperbolicLCS} can be specified in terms of the vectors
$\xi_{1}(x_{0})$, $\xi_{3}(x_{0})$ (or $\text{\ensuremath{\eta}}_{1}(x_{1})$,
$\text{\ensuremath{\eta}}_{3}(x_{1})$). (For a proof, see \citep{Blazevski2014},
Appendix C.)
\begin{prop}
\label{prop:repellingLCS} A smooth material surface is a repelling
hyperbolic LCS if its time-$t_{0}$ position is everywhere normal
to the direction $\xi_{3}$ of largest stretching in forward time;
or, if its time-$t_{1}$ position is everywhere normal to the direction
$\eta_{3}$ of least stretching in backward time. 
\end{prop}

\begin{prop}
\label{prop:attractingLCS}A smooth material surface is an attracting
hyperbolic LCS if its time-$t_{0}$ position is everywhere normal
to the direction $\xi_{1}$ of least stretching in forward time; or,
if its time-$t_{1}$ position is everywhere normal to the direction
$\eta_{1}$ of largest stretching in backward time. 
\end{prop}

Elliptic LCSs (cf. Definitions \ref{def:shearLCS}, \ref{def:ellipticLCS})
can be constructed similarly in terms of the $\xi_{1,2,3}(x_{0})$
(or $\text{\ensuremath{\eta}}_{1,2,3}(x_{1}))$ and the $\sigma_{1,2,3}(x_{0})$:
\begin{prop}
\label{prop:ellipticLCS_old} A smooth material surface is pointwise
shear-maximizing if its time-$t_{0}$ position is everywhere normal
to one of the two directions
\begin{equation}
\tilde{n}{}^{\pm}=\tilde{\alpha}(\sigma_{1},\sigma_{2},\sigma_{3})\,\xi_{1}\pm\tilde{\gamma}(\sigma_{1},\sigma_{2},\sigma_{3})\,\xi_{3}.\label{eq:shear_normal}
\end{equation}
Here $\tilde{\alpha},\,\tilde{\gamma}$ are positive functions of
the singular values $\sigma_{1,2,3}$. (See \citep{Blazevski2014}
for the specific expressions for $\tilde{\alpha}$ and $\tilde{\gamma}$.)\end{prop}
\begin{proof}
See \citep{Blazevski2014}, Theorem 1. 
\end{proof}

\begin{prop}
\label{prop:ellipticLCS}A smooth material surface is nearly uniformly
stretching if its time-$t_{0}$ position is everywhere normal to one
of the two directions
\begin{equation}
n_{\lambda}^{\pm}=\alpha(\sigma_{1},\sigma_{2},\sigma_{3},\lambda)\,\xi_{1}\pm\gamma(\sigma_{1},\sigma_{2},\sigma_{3},\lambda)\,\xi_{3}.\label{eq:elliptic_normal}
\end{equation}
Here $\alpha,\,\gamma$ are positive functions of the singular values
$\sigma_{1,2,3}$, and $\lambda\in[\sigma_{2}(1-\Delta),\sigma_{2}(1+\Delta)]$
with $0\leq\Delta\ll1$. (See \citep{Oettinger2016} for the specific
expressions for $\alpha$ and $\gamma$.)\end{prop}
\begin{proof}
See \citep{Oettinger2016}, Proposition 1.
\end{proof}

\section{Main result: An autonomous dynamical system for all Lagrangian coherent
structures in 3D\label{sec:Uniform3d}}

As reviewed in Sec. \ref{sec:Review3d}, all known LCSs in three dimensions
are geometrically constrained by the singular vectors of the deformation
gradient: Repelling hyperbolic LCSs are normal to the largest singular
vector $\xi_{3}$ (Proposition \ref{prop:repellingLCS}); attracting
hyperbolic LCSs normal to the smallest singular vector $\xi_{1}$
(Proposition \ref{prop:attractingLCS}); elliptic LCSs can be obtained
as surfaces normal to certain linear combinations of $\xi_{1}$ and
$\xi_{3}$ (Propositions \ref{prop:ellipticLCS_old}, \ref{prop:ellipticLCS}).
All these definitions, therefore, pick out material surfaces $\mathcal{M}(t)$
which, at the initial time $t_{0}$, are perpendicular to a normal
field $n$ of the general form
\begin{equation}
n=a\xi_{1}+c\xi_{3},\label{eq:lincombination}
\end{equation}
with real functions $a$ and $c$. In other words, any initial LCS
surface $\mathcal{M}(t_{0})$ is normal to a linear combination of
the smallest and largest singular vector of $DF_{t_{0}}^{t_{1}}$.
Consequently, the intermediate singular vector $\xi_{2}$ must always
lie in the surface $\mathcal{M}(t_{0})$. This means, $\mathcal{M}(t_{0})$
is necessarily tangent to the $\xi_{2}$-direction field. An integral
curve of the $\xi_{2}$-direction field launched from an arbitrary
point of the surface $\mathcal{M}(t_{0})$ will, therefore, remain
confined to $\mathcal{M}(t_{0})$ upon further integration. In the
language of dynamical systems theory, we summarize this observation
as follows (cf. Fig. \ref{fig:Xi2dualCartoon}):
\begin{thm}
\label{thm:xi2_dual}The initial position $\mathcal{M}(t_{0})$ of
any hyperbolic LCS (Definition \ref{def:hyperbolicLCS}) or any elliptic
LCS (Definitions \ref{def:shearLCS}, \ref{def:ellipticLCS}) is an
invariant manifold of the autonomous dynamical system
\begin{equation}
x_{0}^{\prime}=\xi_{2}(x_{0}).\label{eq:xi2_ODE}
\end{equation}
Similarly, final positions $\mathcal{M}(t_{1})$ of hyperbolic and
elliptic LCSs are invariant manifolds of the autonomous dynamical
system 
\begin{equation}
x_{1}'=\eta_{2}(x_{1}).\label{eq:eta2_ODE}
\end{equation}
  
\begin{figure}[h]
\centering{}\includegraphics[width=0.7\columnwidth]{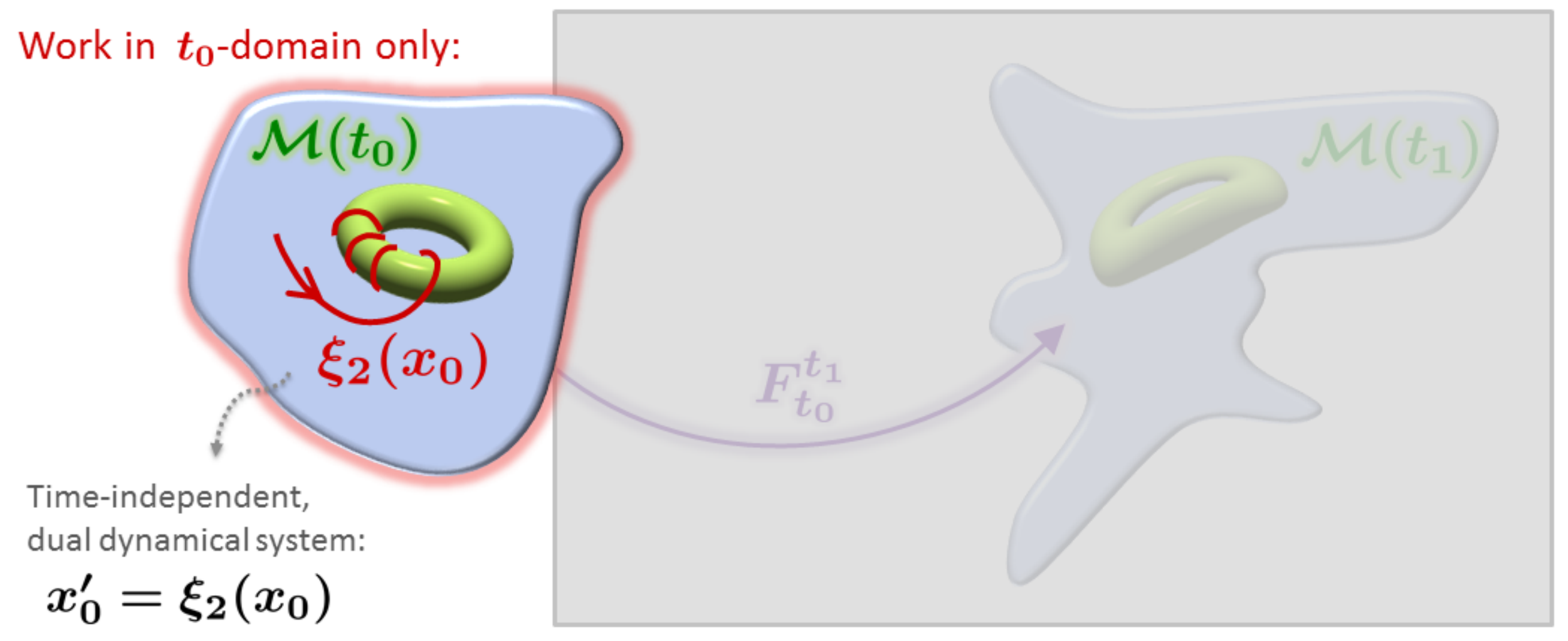}\protect\caption{Schematic of an elliptic LCS $\text{\ensuremath{\mathcal{M}}}(t)$,
revealed as a toroidal invariant manifold $\text{\ensuremath{\mathcal{M}}}(t_{0})$
of the autonomous dual dynamical system (\ref{eq:xi2_ODE}), cf. Theorem
\ref{thm:xi2_dual}.}
\label{fig:Xi2dualCartoon}
\end{figure}

\end{thm}

We refer to the autonomous systems (\ref{eq:xi2_ODE})--(\ref{eq:eta2_ODE})
as the \emph{dual dynamical systems} associated with the original,
non-autonomous system (\ref{eq:flowdef}) over the time interval $[t_{0},t_{1}]$.
The dynamics of these dual systems are not equivalent to the non-autonomous
dynamical system (\ref{eq:flowdef}). Rather, the dual systems allow
locating the LCSs associated with (\ref{eq:flowdef}) using classical
methods for autonomous dynamical systems (e.g., Poincaré maps). 

Since we usually identify LCS surfaces at the initial time $t_{0}$
(cf. Sec. \ref{sec:setup}), we will mostly discuss the $\xi_{2}$-system
(\ref{eq:xi2_ODE}). Analogous results hold for the $\eta_{2}$-system
(\ref{eq:eta2_ODE}). 
\begin{rem}
We refer to the right-hand side of (\ref{eq:xi2_ODE}) as the \emph{$\xi_{2}$-field},
to its integral curves as \emph{$\xi_{2}$-lines}, and to its invariant
manifolds as \emph{$\xi_{2}$-invariant manifolds}. Calling (\ref{eq:xi2_ODE})
a dual dynamical system guides our intuition, but requires some clarification:
For (\ref{eq:xi2_ODE}) to be well-defined, we need to locally assign
an orientation to the $\xi_{2}$-direction field. Along integral curves,
once we assign an initial orientation, this can always be done in
a smooth fashion (cf. Appendix \ref{sec:details-examples}). With
this prescription, the orientation of trajectories in the $\xi_{2}$-system
is defined unambiguously. (Since the $\xi_{2}$-vectors in (\ref{eq:xi2_ODE})
are unit vectors, here, the evolutionary variable is arclength.)
\end{rem}

Theorem \ref{thm:xi2_dual} enables locating unknown LCSs of all types
using only one equation: Any two-dimensional invariant manifold $\mathcal{S}(t_{0})$
of the $\xi_{2}$-system (\ref{eq:xi2_ODE}) is a surface that fulfills
a necessary condition (i.e., tangency to $\xi_{2}$) required for
the initial positions $\mathcal{M}(t_{0})$ of both hyperbolic and
elliptic LCSs. Since invariant manifolds of (\ref{eq:xi2_ODE}) are
already exceptional objects by themselves, any $\xi_{2}$-invariant
manifold $\mathcal{S}(t_{0})$ that we obtain for a given dynamical
system (\ref{eq:flowdef}) is a relevant candidate for an LCS surface
$\mathcal{M}(t_{0})$.

Since the LCS normals from Propositions \ref{prop:repellingLCS}--\ref{prop:ellipticLCS}
do not encompass all linear combinations of $\xi_{1}$ and $\xi_{3}$,
the converse of Theorem \ref{thm:xi2_dual} does not hold. In other
words, a $\xi_{2}$-invariant manifold $\mathcal{S}(t_{0})$ does
not necessarily correspond to an LCS $\mathcal{M}(t_{0})$.  To fully
determine whether $\mathcal{S}(t_{0})$ does satisfy one of the Definitions
\ref{def:hyperbolicLCS}--\ref{def:ellipticLCS}, therefore, one has
to verify tangency to a second vector field (cf. Appendix \ref{sec:perturbations}).
 In applications, however, it is enough to categorize an LCS candidate
qualitatively as either elliptic, hyperbolic repelling or attracting.
As seen in the examples below (cf. Sec. \ref{sec:Examples}), we can
then omit the procedure in Appendix \ref{sec:perturbations} and examine
both the topology of an LCS candidate $\mathcal{S}(t_{0})$ and its
image under the flow map, $\mathcal{S}(t_{1})$, to assess if the
material surface $\mathcal{S}(t)$ belongs to any of the three general
LCS types: Any tubular surface $\mathcal{S}(t_{0})$ is a candidate
for an elliptic LCS, any sheet-like surface $\mathcal{S}(t_{0})$
is a candidate for a hyperbolic LCSs. Mapping $\mathcal{S}(t_{0})$
under the flow map reveals if $\mathcal{S}(t)$ indeed holds up as
an elliptic or hyperbolic LCS.

As outlined in Sec. \ref{sec:Introduction}, previous approaches \citep{Blazevski2014,Oettinger2016}
locate LCSs of all the types in three dimensions (Definitions \ref{def:hyperbolicLCS}--\ref{def:ellipticLCS})
using the expressions for their surface normals from Propositions
\ref{prop:repellingLCS}--\ref{prop:ellipticLCS}. Specifically, these
methods sample the flow domain using extended families of two-dimensional
reference planes. Taking the cross product between the LCS normal
and the normal of each reference plane then defines two-dimensional
direction fields to which the unknown LCS surfaces need to be tangent.
These two-dimensional fields depend on the type of LCS; in particular,
for the near-uniformly stretching LCSs, by (\ref{eq:elliptic_normal}),
there are two parametric families of normal fields $n_{\lambda}^{\pm}$,
which need to be sampled using a dense set of $\lambda$-parameters.
Overall, therefore, one has to perform integrations of a large number
of two-dimensional direction fields. (E.g., \citep{Oettinger2016}
obtained elliptic LCSs in the steady Arnold-Beltrami-Childress from
integral curves of 1600 distinct direction fields.) Accordingly, this
procedure typically produces a large collection of possible intersection
curves between reference planes and LCSs. As a second step, these
approaches require identification of curves from this collection that
can be interpolated into LCS surfaces. Despite these efforts, the
previous approaches \citep{Blazevski2014,Oettinger2016} do not enforce
Theorem \ref{thm:xi2_dual} and hence cannot guarantee more accurate
LCS results than the present approach. An advantage is, however, that
these approaches \citep{Blazevski2014,Oettinger2016} inherently distinguish
between the specific normal fields given in Propositions \ref{prop:repellingLCS}--\ref{prop:ellipticLCS}
and hence do not require further analysis to determine the LCS type.

Clearly, opposed to the previous methods \citep{Blazevski2014,Oettinger2016}
described above, analyzing the $\xi_{2}$-system (\ref{eq:xi2_ODE})
is a conceptually simpler approach to obtaining LCSs in three dimensions:
First, the $\xi_{2}$-field is a single direction field suitable for
all types of LCSs. Secondly, as opposed to considering a large number
of independent two-dimensional equations, the $\xi_{2}$-system (\ref{eq:xi2_ODE})
is defined on a three-dimensional domain. In comparison to the methods
in \citep{Blazevski2014,Oettinger2016}, this eliminates the effort
of handling large amounts of unutilized data and eliminates possible
issues with the placement of reference planes. A full determination
of the LCS types, however, requires verifying tangency to a second
vector field (cf. Appendix \ref{sec:perturbations}).     

In two dimensions, initial positions of LCSs can be viewed as invariant
manifolds of differential equations similar to (\ref{eq:xi2_ODE}).
There, however, the available LCS types (hyperbolic, parabolic and
elliptic LCSs \citep{Haller2011a,Farazmand2014a,Haller2013}) do not
satisfy a single common differential equation: With only two right-singular
vectors $\tilde{\xi}{}_{1,2}$ in two dimensions (and no counterpart
to the intermediate eigenvector $\xi_{2}$ in three dimensions), the
initial positions of hyperbolic and parabolic LCSs are defined by
integral curves of either $\tilde{\xi}_{1}$ or $\tilde{\xi}_{2}$
\citep{Haller2011a,Farazmand2014a}. Similarly, elliptic LCSs are
limit cycles of direction fields belonging to a parametric family
of linear combinations of $\tilde{\xi}_{1}$ and $\tilde{\xi}_{2}$
\citep{Haller2013}. Therefore, there cannot be a counterpart to Theorem
\ref{thm:xi2_dual} in two dimensions. Locating the LCSs in two dimensions
requires analyzing all these differential equations separately. 

In four dimensions and higher, there are no suitable extensions to
the LCS definitions from Sec. \ref{sec:Review3d}, and hence there
is no counterpart to Theorem \ref{thm:xi2_dual} either (cf. Appendix
\ref{sec:Higher-dimensions}).

\section{Examples \label{sec:Examples}}

In this section, we consider several (steady and time-aperiodic) flows
and locate their LCSs by finding invariant manifolds of their associated
$\xi_{2}$-fields.  Our approach is to run long $\xi_{2}$-trajectories
which may asymptotically accumulate on normally attracting invariant
manifolds of the $\xi_{2}$-field (for numerical details, see Appendix
\ref{sec:details-examples}). By Theorem \ref{thm:xi2_dual}, such
invariant manifolds are candidates for time $t_{0}$-positions of
LCSs. Obtaining the LCSs as attractors in the $\xi_{2}$-system ensures
their robustness, whereas this property does not generally hold for
them in the original non-autonomous system. (For incompressible flows,
such as the examples in this section, there are no attractors at all.)

For a generally applicable numerical algorithm, a more refined method
for obtaining two-dimensional invariant manifolds in three-dimensional,
autonomous dynamical systems needs to be combined with the ideas presented
here (cf. Sec. \ref{sec:Conclusions}). We postpone these additional
steps to future work. 

We first consider steady examples where transport barriers are known
from other approaches, and hence the results obtained from the $\xi_{2}$-system
are readily verified. We then move on to an example with a temporally
aperiodic velocity field.

\subsection{Cat's eye flow\label{sub:Cats}}

In Cartesian coordinates $(x,y,z)$, consider a vector field
\begin{equation}
u(x,y,z)=\left(\begin{array}{c}
-\partial_{y}\psi(x,y)\\
\partial_{x}\psi(x,y)\\
W\circ\psi(x,y)
\end{array}\right),\label{eq:two-and-a-half-dimensional}
\end{equation}
where $W$, $\psi$ are smooth, real-valued functions, and $\psi$
is a stream function, i.e., $\Delta\psi=F(\psi)$ for some smooth
function $F$. Any velocity field $u$ satisfying (\ref{eq:two-and-a-half-dimensional})
is a solution of the Euler equations of fluid motion in three dimensions
\citep{majda2002xi2dual}. We consider the two-and-a-half-dimensional
Cat's eye flow \citep{majda2002xi2dual}, given by (\ref{eq:two-and-a-half-dimensional})
with $W(\psi)=\exp(\psi)$ and 
\begin{equation}
\psi(x,y)=-\log[c\,\cosh(y)+\sqrt{c^{2}-1}\cos(x)],\quad c=2.\label{eq:cat-psi}
\end{equation}
We assume that $u=u(x,y,z)$ is defined on the cylinder $S^{1}\times\mathbb{R}^{2}$,
with $x\in\left[0,2\text{\ensuremath{\pi}}\right)$. Because $u$
only depends on the $x,y$-coordinates here, i.e., $u=u(x,y)$, any
flow generated by a velocity field $u$ as in (\ref{eq:two-and-a-half-dimensional})
is called two-and-a-half-dimensional.

Denoting the trajectory passing through $(x_{0},y_{0},z_{0})$ at
time $t_{0}$ by $(x(t),y(t),z(t))$, the flow map takes the form
$F_{t_{0}}^{t_{1}}(x_{0},y_{0},z_{0})=(x_{0},y_{0},z_{0})^{T}+\int_{t_{0}}^{t_{1}}u(x(s),y(s))ds$.
Thus, the flow map $F_{t_{0}}^{t_{1}}$ is linear in $z_{0}$. Consequently,
the deformation gradient $DF_{t_{0}}^{t_{1}}$, its singular values
$\sigma_{1,2,3}$, and singular vectors $\xi_{1,2,3}$ do not depend
on $z_{0}$.

Identifying the coordinates of the domain $D$ of initial positions
$(x_{0},y_{0},z_{0})$ with $(x,y,z)$\textcolor{black}{, we we cannot
expect, however, that any of the $\xi_{1,2,3}$-fields will have a
vanishing $z$-component, i.e., be effectively two-dimensional. }\textbf{\textcolor{red}{}}

For the numerical integrations of the $\xi_{2}$-field (\ref{eq:xi2_ODE}),
we choose 20 representative initial conditions $p_{0}$ in the plane
$z=0$ and, imposing the initial orientation such that the $z$-component
of $\xi_{2}(p_{0})$ is positive, we compute $\xi_{2}$-lines up to
arclength $s=500$. As the time-interval, we consider $[t_{0},t_{1}]=[0,100]$.
We show the results in Fig. \ref{fig:Cat_tf100-2}, together with
level sets of $\psi$ that correspond to the values $\psi(p_{0})$.
\begin{figure}[h]
\includegraphics[width=0.45\textwidth]{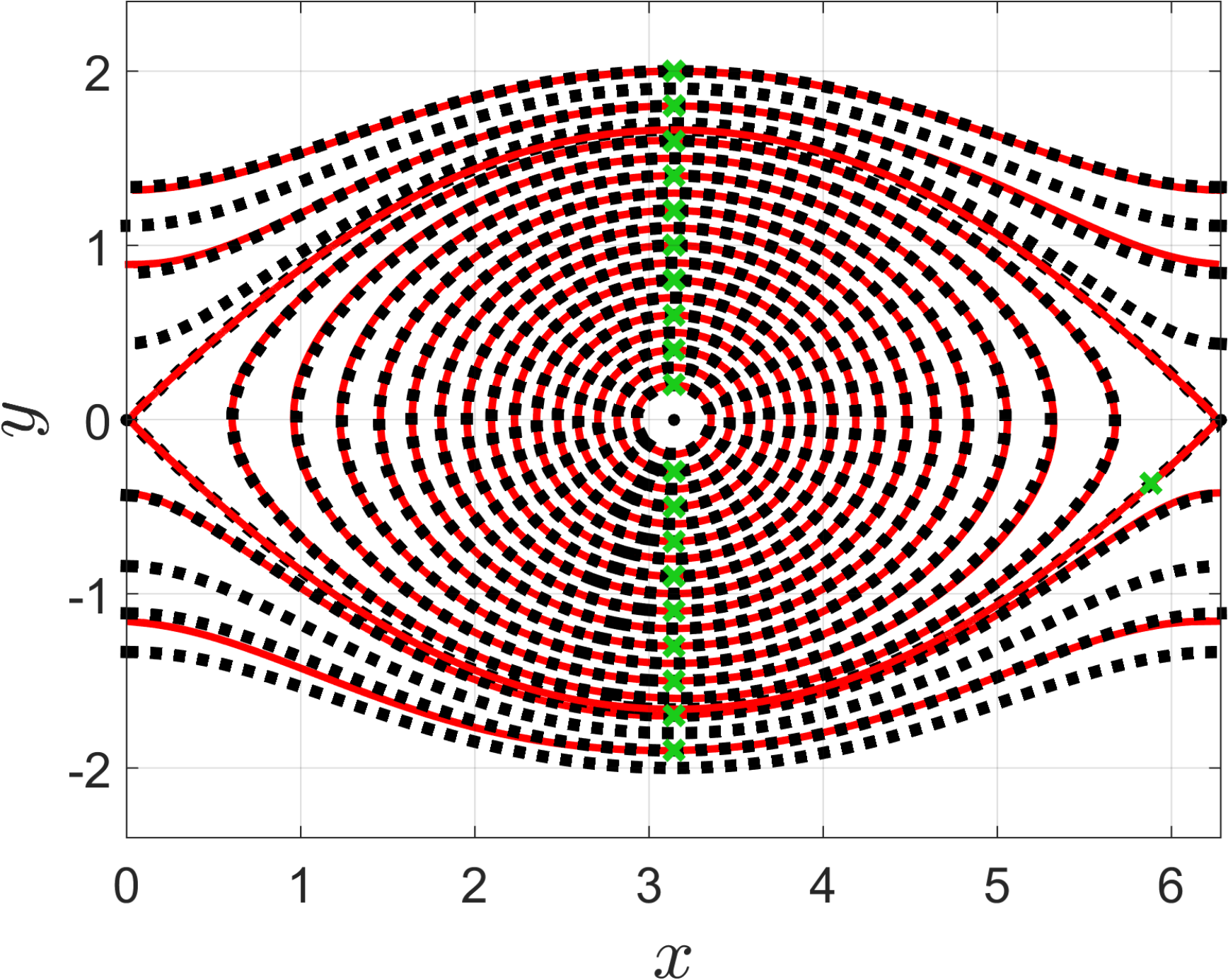}\protect\caption{Cat's eye flow: Comparison between $x,y$-projections of $\xi_{2}$-lines,
displayed for arclength $s\in[0,500]$, (solid red curves) and level
sets of the stream function $\psi$ (dotted black curves). The $\xi_{2}$-lines
have nonzero $z$-components and are confined to generalized cylinders.
The initial conditions of the $\xi_{2}$-lines, $p_{0}$, are marked
by green crosses. \label{fig:Cat_tf100-2} }
\end{figure}
 Each level set of $\psi$ defines a two-dimensional invariant manifold
of the Cat's eye flow. The $\xi_{2}$-lines are well-aligned with
the corresponding level sets of $\text{\ensuremath{\psi}}$, including
the separatrix, showing consistency between the possible locations
of LCSs and the invariant manifolds of the Cat's eye velocity field.
(We note that full alignment would require sampling the infinite-time
dynamics of the Cat's eye flow, i.e., letting $t_{1}\rightarrow\infty$
\citep{Haller2015}.) We observe that the $x,y$-projection of each
$\xi_{2}$-line is a periodic orbit, and thus, each $\xi_{2}$-line
is confined to a generalized (two-dimensional) cylinder.

\subsection{Steady ABC flow\label{sub:ABCsteady}}

Our second steady example is a fully three-dimensional solution of
the Euler equations, the steady Arnold-Beltrami-Childress (ABC) flow
\begin{equation}
u(x,y,z)=\left(\begin{array}{c}
A\sin(z)+C\cos(y)\\
B\sin(x)+A\cos(z)\\
C\sin(y)+B\cos(x)
\end{array}\right),\label{eq:ABC}
\end{equation}
with $A=\sqrt{3},\, B=\sqrt{2},\, C=1$. The coordinates in (\ref{eq:ABC})
are Cartesian, with $(x,y,z)\in[0,2\pi]^{3}$ and periodic boundary
conditions imposed in $x$, $y$ and $z$. 

Using the plane $z=0$ as a Poincaré section, and placing in it a
square grid of $20\times20$ initial positions (cf. Fig. \ref{fig:ABC_Pmaps_a}),
we integrate trajectories of (\ref{eq:ABC}) from time 0 to time $2\cdot10^{4}$.
Retaining only their long-time behavior from the time interval $[10^{4},2\cdot10^{4}]$,
we obtain a large number of iterations of the Poincaré map (cf. Fig.
\ref{fig:ABC_Pmaps_b}). The plot reveals 5 vortical regions surrounded
by a chaotic sea. Each of the vortical regions contains a family of
invariant tori that act as transport barriers.

Here we want to obtain both elliptic and hyperbolic LCSs using the
dual $\xi_{2}$-system (\ref{eq:xi2_ODE}) for $[t_{0},t_{1}]=[0,10]$.
The phase space of the $\xi_{2}$-system coincides with the domain
of (\ref{eq:ABC}). In contrast to trajectories of $u$, independently
of the time interval $[t_{0},t_{1}]$, we can run $\xi_{2}$-lines
as long as we need. Choosing the same Poincaré section and the same
grid of initial conditions as above (cf. Fig. \ref{fig:ABC_Pmaps_a}),
we integrate $\xi_{2}$-lines (initially aligned with $(0,0,1)$)
up to arclength $5\cdot10^{4}$. Retaining segments from the arclength
interval $[4\cdot10^{4},5\cdot10^{4}]$, and intersecting these segments
with the $z=0$ plane, we obtain iterations of a dual Poincaré map
(cf. Fig. \ref{fig:ABC_Pmaps_d}). 
\begin{figure}[h]
\includegraphics[width=0.95\linewidth]{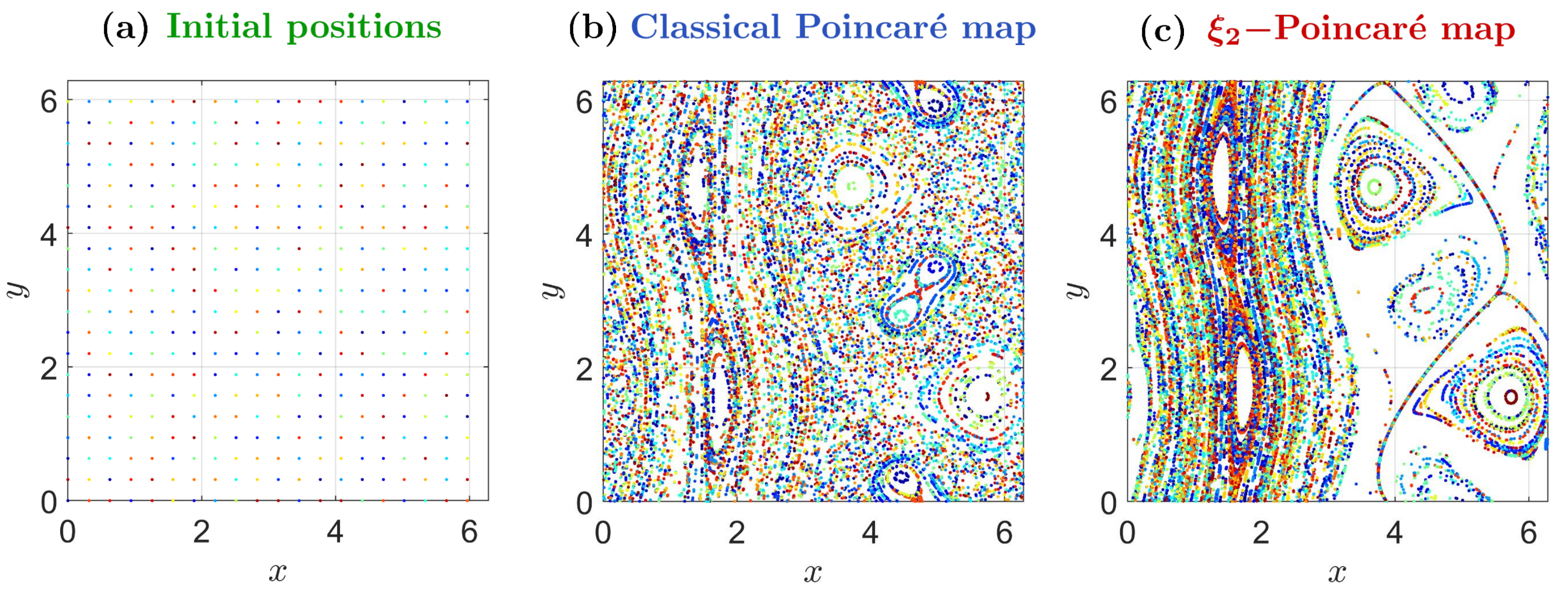}\subfloat{\label{fig:ABC_Pmaps_a}}\subfloat{\label{fig:ABC_Pmaps_b}}\subfloat{\label{fig:ABC_Pmaps_d}}\protect\caption{Steady ABC flow: Comparison of Poincaré maps at $z=0$. (a)\textbf{
}Grid of 20$\times$20 initial positions in the $z=0$-plane.\textbf{
}(b) Poincaré map of (\ref{eq:ABC}) obtained from trajectories over
$[10^{4},2\cdot10^{4}]$, indicating invariant manifolds of the ABC
flow; \textbf{} (c) Poincaré map of the $\xi_{2}$-field, obtained
from $\xi_{2}$-lines over the arclength interval $[4\cdot10^{4},5\cdot10^{4}]$,
indicating initial positions of LCSs. }
\end{figure}
 This Poincaré map indicates invariant manifolds of the dual $\xi_{2}$-system.
Specifically, the principal vortices of the ABC flow correspond to
families of invariant tori of the $\xi_{2}$-field (cf. Fig. \ref{fig:ABC_Pmaps_d}),
which are candidates for initial positions of elliptic LCSs. The tori
of the $\xi_{2}$-system are similar to the invariant tori obtained
from the classical Poincaré map (cf. Fig. \ref{fig:ABC_Pmaps_b}).
In the region corresponding to the chaotic sea, however, the $\xi_{2}$-field
is strongly dissipative and thus reveals a candidate for a transport
barrier in the ABC flow that has no counterpart in the classical Poincaré
map obtained from the asymptotic dynamics of the incompressible system
(\ref{eq:ABC}): We see a structure that has a large basin of attraction
in the dual dynamics of the $\xi_{2}$-system and, secondly, spans
the entire domain. In Sec. \ref{sub:ABCaperiodic5}, we will examine
a slightly perturbed version of this structure in detail, finding
that it is a hyperbolic repelling LCS. 

We note that computing Poincaré maps for the $\xi_{2}$-system does
not imply applying the flow map $F_{t_{0}}^{t_{1}}$ repetitively.
Iterating a $\xi_{2}$-based Poincaré map simply serves to refine
our understanding of the LCSs associated with $F_{t_{0}}^{t_{1}}$.
Indeed, the iterated Poincaré map highlights intersections of fixed
LCSs with a given plane of the $\xi_{2}$-system in more and more
detail.

\subsection{Time-aperiodic ABC-type flow\label{sub:ABCaperiodic5}}

We next use the dual $\xi_{2}$-system (\ref{eq:xi2_ODE}) to analyze
a time-aperiodic modification of the ABC flow, given by (\ref{eq:ABC})
with the replacements
\begin{equation}
\begin{aligned}B\mapsto\tilde{B}(t) & = & B+B\cdot k_{0}\tanh(k_{1}t)\cos[(k_{2}t)^{2}],\\
C\mapsto\tilde{C}(t) & = & C+C\cdot k_{0}\tanh(k_{1}t)\sin[(k_{3}t)^{2}].
\end{aligned}
\label{eq:ABCaperiodic_funcs}
\end{equation}
Neither a classical Poincaré map nor any other method requiring long
trajectories are options here, due to the temporal aperiodicity of
the system. In (\ref{eq:ABCaperiodic_funcs}), we choose $k_{0}=0.3$,
$k_{1}=0.5$, $k_{2}=1.5$ and $k_{3}=1.8$. We show the functions
$\tilde{B}(t)-B$, $\tilde{C}(t)-C$ in Fig. \ref{fig:ABC_signals}.
\begin{figure}[h]
\includegraphics[width=0.4\textwidth]{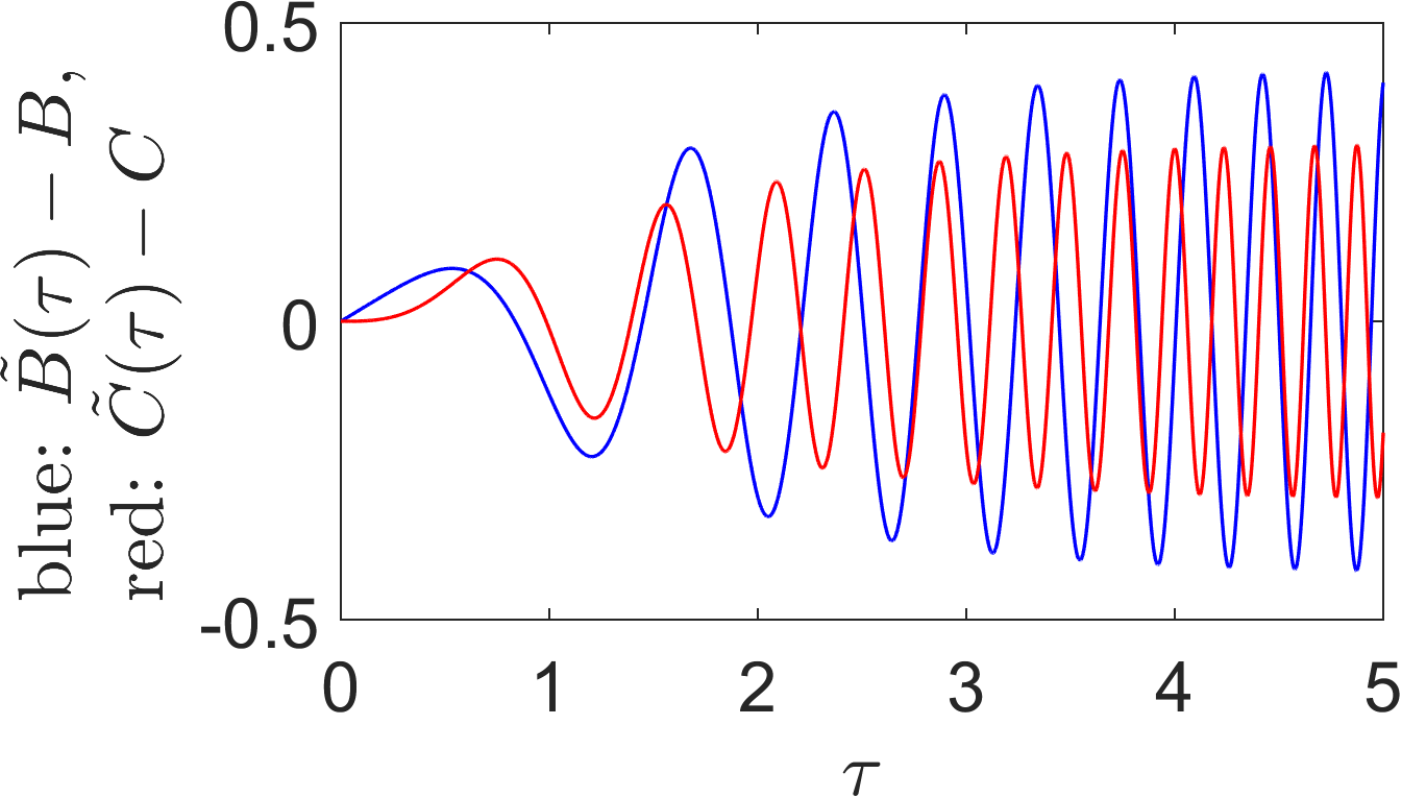}\protect\caption{Time dependence of the coefficient functions $\tilde{B}(t)$, $\tilde{C}(t)$
in (\ref{eq:ABCaperiodic_funcs}).}
\label{fig:ABC_signals}
\end{figure}
 Elliptic LCSs in similar time-aperiodic ABC-type flows have been
obtained in \citep{Blazevski2014,Oettinger2016}; hyperbolic repelling
LCSs in \citep{Blazevski2014}, although only of small extent in the
$z$-direction. 

Considering the $\xi_{2}$-system for the time interval $[t_{0},t_{1}]=[0,5]$,
we compute the dual Poincaré map (cf. Fig. \ref{fig:ABCaperiodic_xi2Pmap}).
The algorithm and numerical settings are the same as in Sec. \ref{sub:ABCsteady}.
\begin{figure}[h]
\includegraphics[width=1\columnwidth]{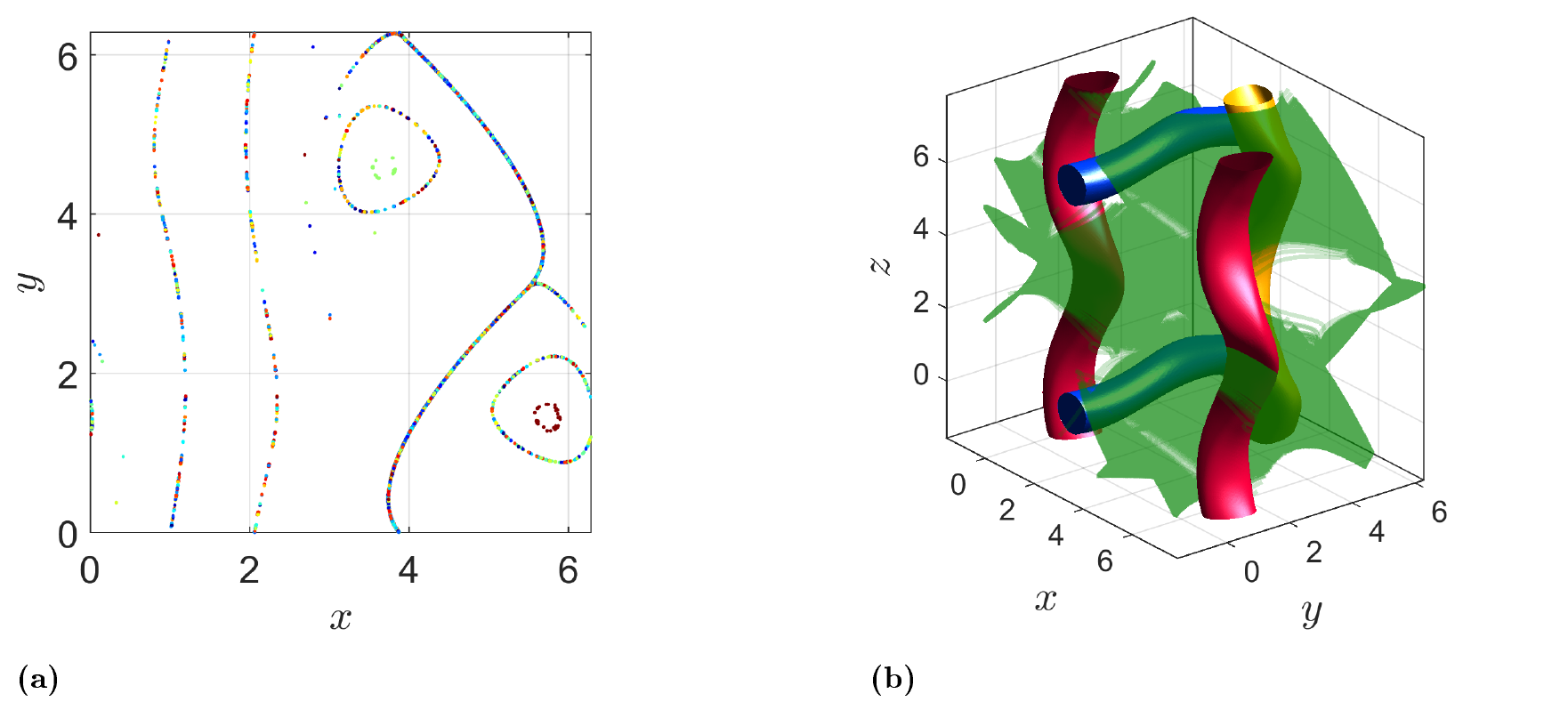}\subfloat{\label{fig:ABCaperiodic_xi2Pmap}}\hfill{}\subfloat{\label{fig:ABCaperiodic_3Dstructures}}\hfill{}\protect\caption{Time-aperiodic ABC-type flow: Arc segments of $\xi_{2}$-lines (corresponding
to arclength $s\in[4\cdot10^{4},5\cdot10^{4}]$) reveal locations
of elliptic and hyperbolic LCSs. (a) Dual Poincaré map, showing intersections
between the Poincaré section $z=0$ and possible time-$t_{0}$ locations
of elliptic and hyperbolic LCSs. (b) Possible time-$t_{0}$ locations
of elliptic and hyperbolic LCSs: The structure in green (indicating
a hyperbolic LCS) consists of segments from several $\xi_{2}$-lines.
The tubular surfaces (indicating elliptic LCSs) are fitted from point
data of individual $\xi_{2}$-lines. Here we use the periodicity of
the phase space to extend the domain slightly beyond $[0,2\pi]^{3}$. }
\end{figure}
 Compared to Fig. \ref{fig:ABC_Pmaps_d}, there are a few structures
that persist under the time-aperiodic perturbation (\ref{eq:ABCaperiodic_funcs})
to the velocity field (\ref{eq:ABC}): The large (presumably hyperbolic)
structure spanning the flow domain is still present and barely changed.
In Fig. \ref{fig:ABCaperiodic_3Dstructures}, we show $\xi_{2}$-lines
corresponding to this hyperbolic LCS candidate (green). The $\xi_{2}$-lines
indicate a complicated surface which they, however, do not cover densely.
Regarding elliptic structures, instead of entire families of $\xi_{2}$-invariant
tori, we are left with three large elliptic structures, each with
a sizable domain of attraction (cf. Fig. \ref{fig:ABCaperiodic_xi2Pmap}).
The $\xi_{2}$-lines corresponding to these elliptic structures yield
tori, which we show as tubular surfaces in Fig. \ref{fig:ABCaperiodic_3Dstructures}
(red, blue, yellow). The dual Poincaré map (Fig. \ref{fig:ABCaperiodic_xi2Pmap})
also shows that, inside two of these tori, there are additional, smaller
elliptic structures. By plotting the $\xi_{2}$-lines corresponding
to these smaller objects (not shown), we find that the surfaces they
indicate are not tori and thus ignore them in our search for LCS candidates. 

In Fig. \ref{fig:ABCaperiodic_ellLCS_t0p0}, we represent the yellow
tubular surface from Fig. \ref{fig:ABCaperiodic_3Dstructures} in
toroidal coordinates 
\begin{equation}
\begin{aligned}\bar{x} & = & (x-x_{c}(z)+R_{1})\cos(z),\\
\bar{y} & = & (x-x_{c}(z)+R_{1})\sin(z),\\
\bar{z} & = & R_{2}(y-y_{c}(z)),
\end{aligned}
\label{eq:toroidal_coords}
\end{equation}
with $R_{1}=2,\, R_{2}=1$. 
\begin{figure}[h]
\includegraphics[width=0.9\columnwidth]{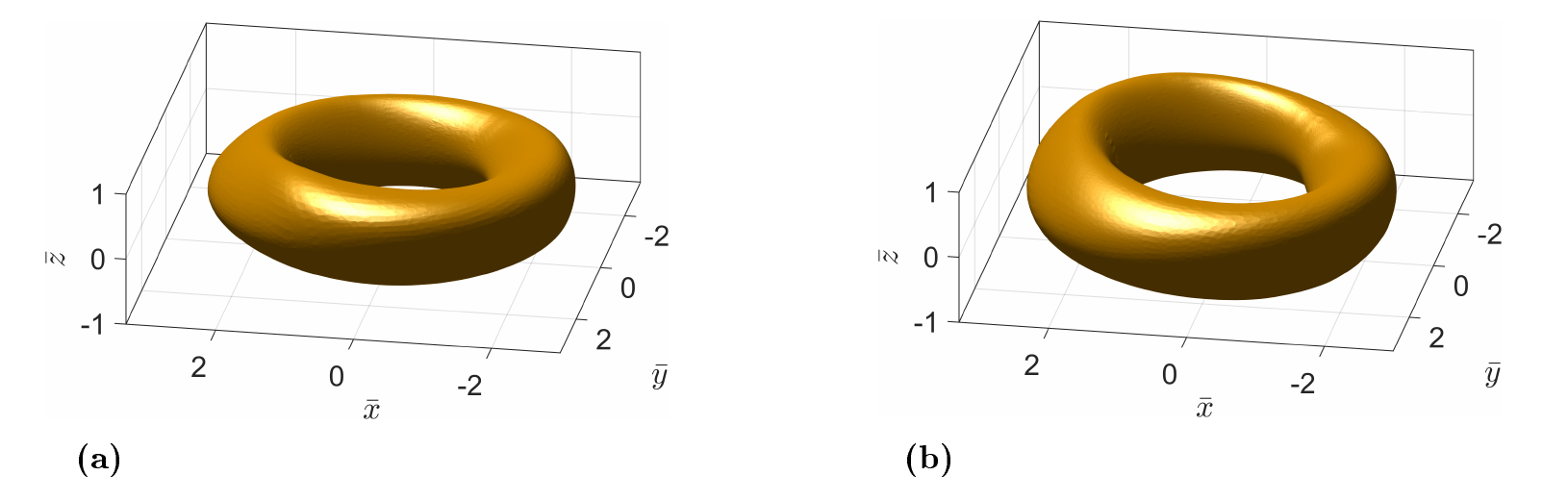}\subfloat{\label{fig:ABCaperiodic_ellLCS_t0p0}}\hfill{}\subfloat{\label{fig:ABCaperiodic_ellLCS_t5p0}}\hfill{}\protect\caption{Time-aperiodic ABC-type flow: Mapping one of the tubular surfaces
obtained from the $\xi_{2}$-lines (cf. Fig. \ref{fig:ABCaperiodic_3Dstructures},
yellow) under the flow map $F_{0}^{5}$, we confirm that this surface
is a useful elliptic LCS. (a) Elliptic LCS surface at time $t_{0}=0$.
(b) Elliptic LCS surface at time $t_{1}=5$. }
\end{figure}
 In (\ref{eq:toroidal_coords}), the functions $x_{c}(z),\, y_{c}(z)$
are the $x,$ $y$ coordinates of the (approximate) vortex center.
(For evaluating $x_{c}(z)$ and $y_{c}(z)$, we use our numerical
values from previous work \citep{Oettinger2016}.) Mapping the resulting
torus under the flow map $F_{0}^{5}$, we see that it does advect
coherently over the interval $[t_{0},t_{1}]=[0,5]$ (cf. Fig. \ref{fig:ABCaperiodic_ellLCS_t5p0}).
Therefore, even though this surface was just obtained from tangency
to $\xi_{2}$ (a necessary condition for Definition \ref{def:ellipticLCS}),
it renders a full-blown elliptic LCS. 

We next examine locally whether the complicated green structure from
Fig. \ref{fig:ABCaperiodic_3Dstructures} indeed corresponds to a
hyperbolic LCS (Definition \ref{def:hyperbolicLCS}): In Fig. \ref{fig:ABCaperiodic_repLCS_t0p0},
we take an illustrative part of the domain and interpolate a surface
from the $\xi_{2}$-lines (green). Centered around a point in the
surface, we additionally place a sphere of tracers (purple). Mapping
the two objects forward in time under the flow map $F_{0}^{1}$, we
see that the tracers deform into an ellipsoid that is most elongated
in the direction normal to the advected surface (cf. Fig. \ref{fig:ABCaperiodic_repLCS_t1p0}).
\begin{figure}[h]
\includegraphics[width=0.7\columnwidth]{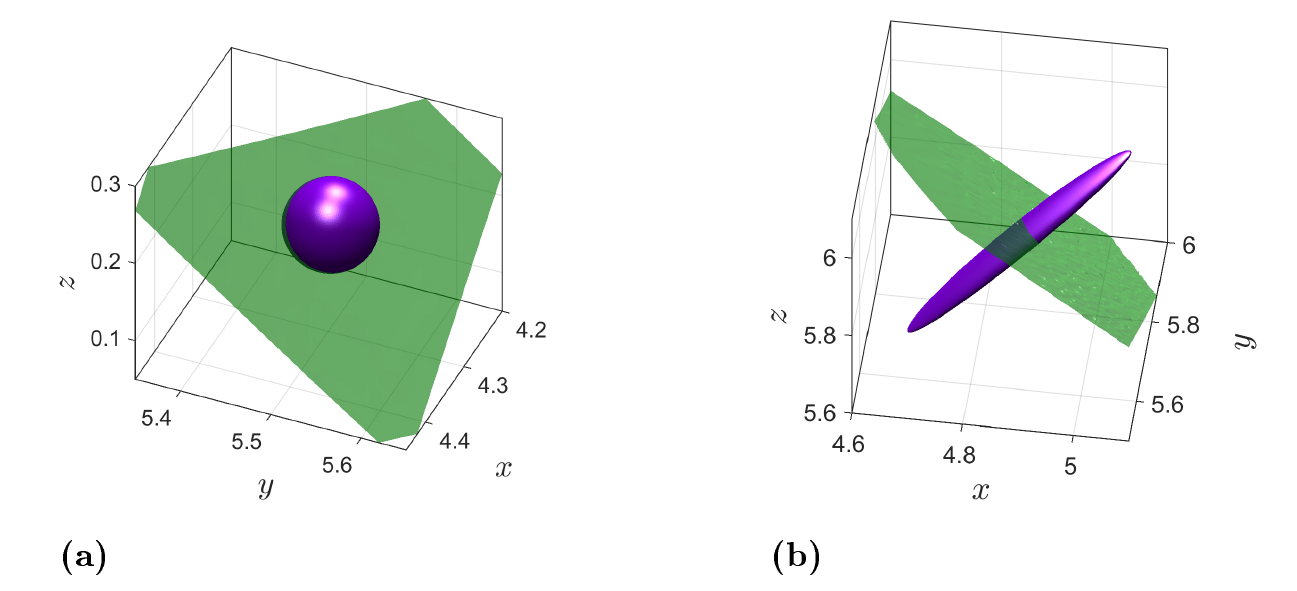}\subfloat{\label{fig:ABCaperiodic_repLCS_t0p0}}\subfloat{\label{fig:ABCaperiodic_repLCS_t1p0}}\protect\caption{ Time-aperiodic ABC-type flow: Local impact of the hyperbolic repelling
LCS surface (interpolated from $\xi_{2}$-lines). (a) Zoom-in on the
hyperbolic repelling LCS surface at time $t_{0}=0$ (green), shown
together with a sphere formed by tracers (purple). (b) Time-1 positions
of the hyperbolic repelling LCS surface and the deformed tracer sphere
(obtained under $F_{0}^{1}$).}
\end{figure}
 Considering Proposition \ref{prop:repellingLCS} and Fig. \ref{fig:svd},
we thus classify this structure as a repelling hyperbolic LCS. (For
an approach to confirming this globally, see Appendix \ref{sec:perturbations}.)
Considering Fig. \ref{fig:ABCaperiodic_3Dstructures}, we see that
this structure is much larger than the hyperbolic LCS obtained for
a similar time-aperiodic ABC-type flow in previous work (cf. \citep{Blazevski2014},
Fig. 15). 

By Theorem \ref{thm:xi2_dual}, we can also take the direction field
$\eta_{2}$ and repeat the above analysis. Using the same algorithm
and numerical parameters as for the previous $\xi_{2}$-Poincaré map
(cf. Fig. \ref{fig:ABCaperiodic_xi2Pmap}), except that we now take
the backward-time flow map $F_{5}^{0}$ instead of $F_{0}^{5}$, we
obtain a Poincaré map for the dual dynamical system $x_{1}'=\eta_{2}(x_{1})$
(cf. Fig. \ref{fig:ABCaperiodic_eta2Pmap}). 
\begin{figure}[h]
\includegraphics[width=0.4\textwidth]{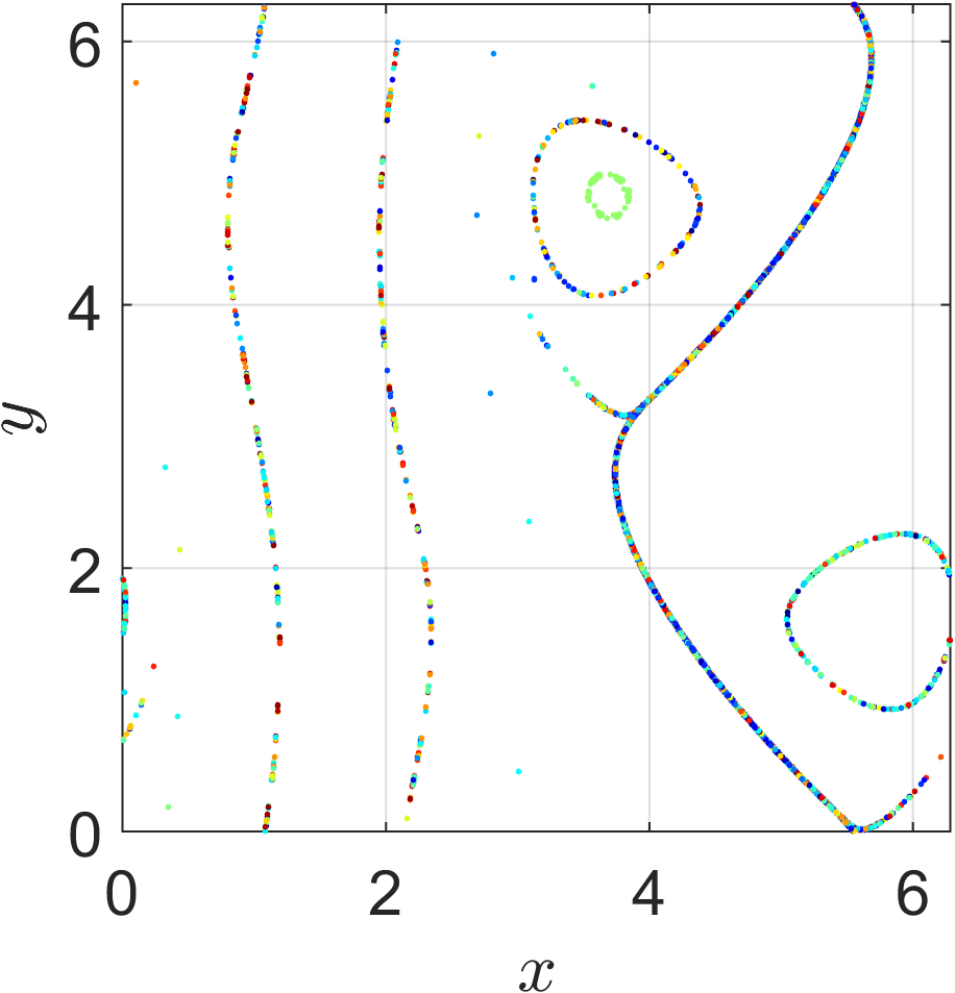}\protect\caption{Dual Poincaré map obtained from $x_{1}'=\eta_{2}(x_{1})$, showing
intersections between the Poincaré section $z=0$ and possible time-$t_{1}$
locations of elliptic and hyperbolic LCSs. }
\label{fig:ABCaperiodic_eta2Pmap}
\end{figure}
 This Poincaré map reveals possible time-$t_{1}$ positions of LCSs.
The result is similar to the $\xi_{2}$-Poincaré map (cf. Fig. \ref{fig:ABCaperiodic_xi2Pmap}),
showing again a large hyperbolic structure, and the time-$t_{1}$
positions of the tori obtained earlier (cf. Fig. \ref{fig:ABCaperiodic_3Dstructures}). 

We perform a local deformation analysis for the large hyperbolic structure
indicated by Fig. \ref{fig:ABCaperiodic_eta2Pmap}: From a sample
part of the $\eta_{2}$-lines corresponding to this structure, we
fit a surface (cf. Fig. \ref{fig:ABCaperiodic_attLCS_t5p0}, colored
green) and map it backward in time under $F_{5}^{4}$, obtaining a
surface at time $t=4$ (cf. Fig. \ref{fig:ABCaperiodic_attLCS_t4p0},
green). 
\begin{figure}[h]
\centering{}\includegraphics[width=0.7\columnwidth]{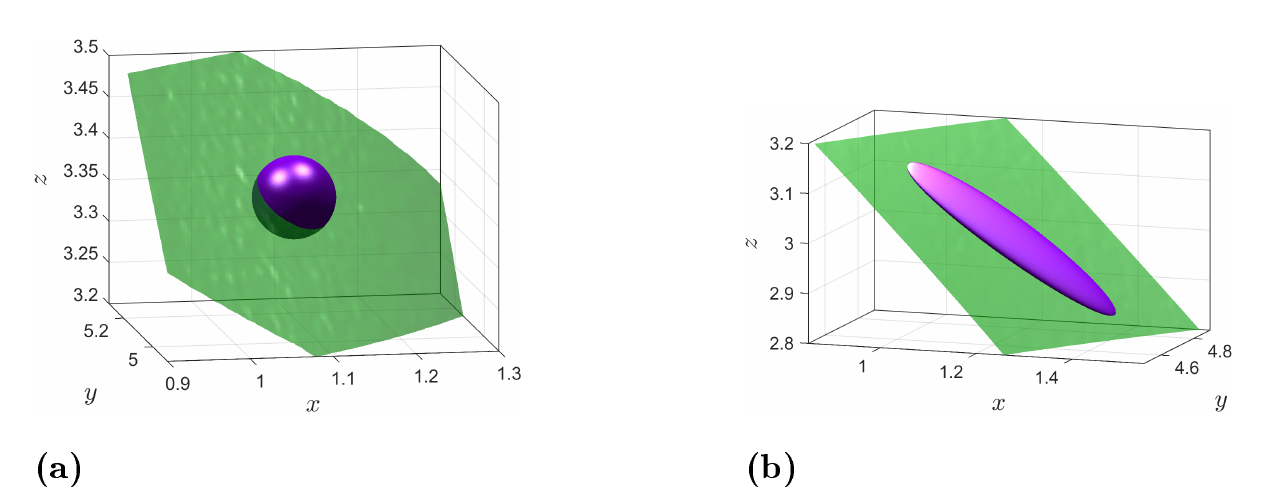}\subfloat{\label{fig:ABCaperiodic_attLCS_t4p0}}\subfloat{\label{fig:ABCaperiodic_attLCS_t5p0}}\protect\caption{Time-aperiodic ABC-type flow: Local impact of the hyperbolic attracting
LCS surface (fitted from $\eta_{2}$-lines). (a) Zoom-in on the hyperbolic
attracting LCS surface at time $t=4$ (green), shown together with
a sphere formed by tracers (purple). (b) Time-$t_{1}$ positions of
the hyperbolic attracting LCS surface and the deformed tracer sphere
(obtained under $F_{4}^{5}$).}
\end{figure}
 Then we place a small tracer sphere (purple) in this part of the
surface. Mapping both the time-4 surface and the tracer sphere forward
in time under $F_{4}^{5}$, we find that the tracers fully align with
the surface (cf. Fig. \ref{fig:ABCaperiodic_attLCS_t5p0}). By Proposition
\ref{prop:attractingLCS} and Fig. \ref{fig:svd}, this suggests that
the large hyperbolic structure from Fig. \ref{fig:ABCaperiodic_eta2Pmap}
belongs to the time-$t_{1}$ position of an attracting hyperbolic
LCS. (For confirming this globally, see Appendix \ref{sec:perturbations}.)
\begin{rem}
\label{rem:forward-repelling}With the present approach, for incompressible
flows, it is generally easier to obtain attracting hyperbolic LCSs
$\mathcal{M}(t_{1})$ at time $t_{1}$, rather than at time $t_{0}$:
An attracting LCS at time $t_{0}$ is a surface $\mathcal{M}(t_{0})$
parallel to $\xi_{2}$ and $\xi_{3}$ (cf. Proposition \ref{prop:attractingLCS}).
Mapping $\mathcal{M}(t_{0})$ to $\mathcal{M}(t_{1})$, the area element
changes by a factor of $\sigma_{2}\sigma_{3}$. Due to incompressibility
($\sigma_{1}\sigma_{2}\sigma_{3}=1$), any attracting LCS is guaranteed
to stretch in forward-time ($\sigma_{2}\sigma_{3}>1$). Since separation
can, e.g., grow exponentially in time ($\sigma_{3}\propto\exp(t_{1}-t_{0})$),
we generally expect the stretching of an attracting LCS to be substantial
($\sigma_{2}\sigma_{3}\gg1$). At the final time $t_{1}$, we thus
expect that any attracting LCS of global impact, $\mathcal{M}(t_{1})$,
traverses a significant portion of the phase space. At time $t_{0}$,
on the other hand, the surface $\mathcal{M}(t_{0})$ can still be
very small. In this sense, seeking LCSs as invariant manifolds of
the $\eta_{2}$-field is generally easier than using the $\xi_{2}$-field.
For repelling LCSs, which shrink between times $t_{0}$ and $t_{1}$,
the converse holds. (In two dimensions, the challenges of computing
repelling and attracting hyperbolic LCSs at different times $t${*}
are similar \citep{Farazmand2013,Karrasch201583}.) 
\end{rem}
In summary, compared to previous methods of identifying LCSs from
various two-dimensional direction fields \citep{Blazevski2014,Oettinger2016},
the advantage of the present approach is that it reveals both hyperbolic
and elliptic LCSs from integrations of a single direction field. Instead
of using multiple one-dimensional Poincaré sections \citep{Blazevski2014,Oettinger2016},
we can therefore search LCSs globally by using two-dimensional Poincaré
sections (cf. Figs. \ref{fig:ABC_Pmaps_d}, \ref{fig:ABCaperiodic_xi2Pmap},
\ref{fig:ABCaperiodic_eta2Pmap}). Finally, as opposed to classical
Poincaré maps that require autonomous or time-periodic systems, the
dual Poincaré map is well-defined for any non-autonomous system.
We in fact treat autonomous, time-periodic and time-aperiodic dynamical
systems on the same footing, while still benefiting from the advantages
that a classical Poincaré map offers.

\section{Conclusions\label{sec:Conclusions} }

We have presented a unified approach to obtaining elliptic and hyperbolic
LCSs in three-dimensional unsteady flows. In contrast to prior methods
based on different direction fields for different types of LCSs \citep{Blazevski2014,Oettinger2016},
we obtain a common direction field, the intermediate eigenvector field,
$\xi_{2}(x_{0})$, of the right Cauchy-Green strain tensor. Initial
positions of all variational LCSs in three dimensions are necessarily
invariant manifolds of this autonomous direction field. Equivalently,
LCS final positions are invariant manifolds of the intermediate eigenvector
field, $\eta_{2}(x_{1})$, of the left Cauchy-Green strain tensor.
We can thus identify LCS surfaces globally by classic methods for
autonomous dynamical systems. While the $\xi_{2}$- and $\eta_{2}$-systems
by themselves do not discriminate between LCS types, the procedure
from Appendix \ref{sec:perturbations} outlines how to numerically
assess the LCS type if needed. 

Overall, the present approach is significantly simpler than previous
numerical methods \citep{Blazevski2014,Oettinger2016}, and reveals
larger hyperbolic LCSs in the time-aperiodic ABC-type flow than seen
in a comparable example from previous work \citep{Blazevski2014}.
An important advantage of our approach is that LCSs are attractors
of the generally dissipative $\xi_{2}$-system, which  is not the
case in the original, typically incompressible system. Obtaining the
LCSs as attractors of the dual $\xi_{2}$-system also guarantees their
structural stability, implying that these structures will persist
under small perturbations to the underlying flow. Our approach is
restricted to three-dimensional systems, which is, however, highly
relevant for fluid mechanical applications.

With the examples of Sec. \ref{sec:Examples}, we have illustrated
the ability of the $\xi_{2}$-system to reveal LCSs. For a broadly
applicable numerical method, further development is required. Computing
two-dimensional invariant manifolds of the $\xi_{2}$-field by simply
running long integral curves is not always efficient. General approaches
for growing global stable and unstable manifolds of autonomous, three-dimensional
vector fields are, however, available in the literature (cf. \citep{krauskopf2005survey}
for a review). We expect that a general computational method for obtaining
LCSs from the $\xi_{2}$-system (\ref{eq:xi2_ODE}) can be most easily
developed by transferring one of these available approaches to computing
invariant manifolds from the setting of vector fields to direction
fields. For a given dynamical system, one would first compute the
$\xi_{2}$-field on a grid, and then apply the most suitable method
for growing invariant manifolds to construct LCSs globally in the
dual $\xi_{2}$-system\@.    

\appendix

\section{For incompressible flows, $\sigma_{2}$ is the singular value of
$DF_{t_{0}}^{t_{1}}$ closest to unity\label{sec:sigma2_unity}}

We clarify our statement that $0<\sigma_{1}<\sigma_{2}<\sigma_{3}$
and incompressibility (i.e., $\sigma_{1}\sigma_{2}\sigma_{3}=1$)
imply that $\sigma_{2}$ is the singular value of $DF_{t_{0}}^{t_{1}}$
closest to unity. We first note that $\sigma_{1}=\sqrt[3]{\sigma_{1}^{3}}<\sqrt[3]{\sigma_{1}\sigma_{2}\sigma_{3}}=1$,
and, similarly, $\sigma_{3}>1$. In general, it is unclear whether
$\sigma_{2}<1$, $\sigma_{2}=1$, or $\sigma_{2}>1$. Due to the inequalities
\begin{equation}
\sigma_{1}<\min\left\{ \sigma_{2},\frac{1}{\sigma_{2}}\right\} \leq1\leq\max\left\{ \sigma_{2},\frac{1}{\sigma_{2}}\right\} <\sigma_{3},\label{eq:sigma2_inequality}
\end{equation}
however, we consider $\sigma_{2}$ as the singular value closest to
unity. Eq. \ref{eq:sigma2_inequality} follows from a more general
statement: 
\begin{lem}
\label{label:lemma_b_mean}Given any three real numbers $a,$ $b,$
and c satisfying $0<a<b<c$, denoting their geometric mean by 
\begin{equation}
m=\sqrt[3]{abc},\label{eq:geom_mean}
\end{equation}
we have

\begin{equation}
\frac{a}{m}<\min\left\{ \frac{b}{m},\frac{m}{b}\right\} \leq1\leq\max\left\{ \frac{b}{m},\frac{m}{b}\right\} <\frac{c}{m}.\label{eq:abc_inequalities}
\end{equation}
\end{lem}
\begin{proof}
Denoting the natural logarithm by $\log$, we introduce $M=\log(m)$,
$A=\log(a),$ $B=\log(b),$ and $C=\log(c)$. Taking the logarithm
of (\ref{eq:geom_mean}), we then obtain
\begin{equation}
3M=A+B+C.\label{eq:log_abc_inequalities}
\end{equation}
Furthermore, since $a=\sqrt[3]{a^{3}}<\sqrt[3]{abc}=m,$ we have
\begin{equation}
M-A>0,\label{eq:MgtrA}
\end{equation}
and, similarly,
\begin{equation}
C-M>0.\label{eq:CgtrM}
\end{equation}

\begin{enumerate}
\item For the first inequality in (\ref{eq:abc_inequalities}), we show
that $a/m<m/b$. By strict monotonicity of the logarithm, this is
equivalent to $\log\left(\tfrac{a}{m}\right)<\log\left(\tfrac{m}{b}\right),$
which we verify as follows:
\[
\log\left(\tfrac{a}{m}\right)=A-M\stackrel{\text{(\ref{eq:log_abc_inequalities})}}{=}3M-B-C-M=(M-B)-(C-M)\stackrel{\text{(\ref{eq:CgtrM})}}{<}M-B=\log\left(\tfrac{m}{b}\right).
\]
For the last inequality in (\ref{eq:abc_inequalities}), we can similarly
show that $m/b<c/m$ (using (\ref{eq:MgtrA}) instead of (\ref{eq:CgtrM})).
\item We show that $\min\left\{ \log\left(\frac{b}{m}\right),\log\left(\frac{m}{b}\right)\right\} \leq0$,
which is equivalent to $\min\left\{ \frac{b}{m},\frac{m}{b}\right\} \leq1$.
To verify the former inequality, we use that the minimum of any two
real numbers $r_{1}$ and $r_{2}$ satisfies $\min\{r_{1},r_{2}\}=\tfrac{r_{1}+r_{2}}{2}-\tfrac{\left|r_{1}-r_{2}\right|}{2}.$
We obtain
\[
\min\left\{ \log\left(\tfrac{b}{m}\right),\log\left(\tfrac{m}{b}\right)\right\} =\tfrac{1}{2}\left[(B-M)+(M-B)\right]-\tfrac{1}{2}\left|(B-M)-(M-B)\right|,
\]
and, thus,
\[
\min\left\{ \log\left(\tfrac{b}{m}\right),\log\left(\tfrac{m}{b}\right)\right\} =-\left|B-M\right|\leq0.
\]
Similarly, we can use $\max\{r_{1},r_{2}\}=\tfrac{r_{1}+r_{2}}{2}+\tfrac{\left|r_{1}-r_{2}\right|}{2}$
and show that $1\leq\max\left\{ \frac{b}{m},\frac{m}{b}\right\} .$
\end{enumerate}

\end{proof}
Setting $a=\sigma_{1}$, $b=\sigma_{2},$ $c=\sigma_{3}$ and $m=1$,
Lemma \ref{label:lemma_b_mean} implies (\ref{eq:sigma2_inequality}).

\section{Theorem \ref{thm:xi2_dual} and higher dimensions\label{sec:Higher-dimensions}}

We discuss the possibility of a counterpart to our main result, Theorem
\ref{thm:xi2_dual}, in higher dimensions. We start with four dimensions,
where there are four singular vectors $\xi_{1,2,3,4}$. As in Sec.
\ref{sec:Review3d}, we label them such that the corresponding singular
values $\sigma_{1,2,3,4}$ are in ascending order. 
\begin{example*}
As a prerequisite, we would need to extend, e.g., the notion of a
hyperbolic repelling LCS (cf. Definition \ref{def:hyperbolicLCS})
from three to four dimensions. As in Proposition \ref{prop:repellingLCS},
we would need a three-dimensional hypersurface $\mathcal{M}(t_{0})$
in $\mathbb{R}^{4}$ which is normal to $\xi_{4}$ and hence tangent
to $\xi_{1,2,3}$ everywhere. It is not a priori obvious whether such
a geometry is possible or not.

Consider a small open ball $B\subset\mathbb{R}^{4}$ where the singular
values $\sigma_{1,2,3,4}$ are distinct. Within $B$, we may assume
that the $\xi_{1,2,3,4}$-fields are smooth vector fields. We denote
the Lie bracket between two such vector fields $v$ and $w$ by $[v,w]$.

 We want to construct a three-dimensional hypersurface $\mathcal{M}(t_{0})$
such that $\mathcal{M}(t_{0})\cap B$ is normal to $\xi_{4}$. This
is possible only if the fields $\xi_{1,2,3}$ satisfy  
\begin{equation}
[\xi_{1},\xi_{2}],[\xi_{1},\xi_{3}],[\xi_{2},\xi_{3}]\in\text{span}\{\xi_{1},\xi_{2},\xi_{3}\}\label{eq:Lie-in-span}
\end{equation}
for all points in $\mathcal{M}(t_{0})\cap B$ (cf., e.g., \citep{Lee2012involutive}).
Conditions (\ref{eq:Lie-in-span}) are equivalent to the Frobenius
conditions  
\begin{equation}
\left\langle [\xi_{1},\xi_{2}],\xi_{4}\right\rangle =0,\,\left\langle [\xi_{1},\xi_{3}],\xi_{4}\right\rangle =0,\,\left\langle [\xi_{2},\xi_{3}],\xi_{4}\right\rangle =0.\label{eq:frobenius}
\end{equation}
(In the context of LCSs, such conditions have already been considered
in \citep{Blazevski2014}.) Unless 0 is a critical value, by the Preimage
Theorem \citep{guillemin1974differential}, each of the three conditions
in (\ref{eq:frobenius}) defines a codimension-one submanifold in
$B$. Now there are two main possibilities:

\emph{Case 1: }We suppose that 0 is a regular value for all conditions
in (\ref{eq:frobenius}). Since the conditions (\ref{eq:frobenius})
are generally independent from each other, the subset $S$ of $B$
where all three conditions are satisfied simultaneously is codimension-three,
i.e., a line. For $\mathcal{M}(t_{0})$ to be a well-defined repelling
LCS, we need $\mathcal{M}(t_{0})\cap B$ to be a subset of $S.$ By
our assumption, however, $\mathcal{M}(t_{0})\cap B$ is a three-dimensional
hypersurface. Since $S$ is only one-dimensional, we have reached
a contradiction. 

\emph{Case 2:} The remaining possibility is that 0 is a critical value
for at least one of the conditions in (\ref{eq:frobenius}). Then
there is no general restriction on the geometry of the corresponding
zero-level sets from (\ref{eq:frobenius}). In particular, if 0 is
critical value for at least two of the three conditions in (\ref{eq:frobenius}),
then the subset $S$ of $B$ where all three conditions are satisfied
simultaneously can be a three- or four-dimensional manifold. In this
case, $S$ can contain a three-dimensional surface $\mathcal{M}(t_{0})\cap B$
and, thus, locally allow for a repelling LCS $\mathcal{M}(t_{0})$.
The catch is, however, that the set of critical values for each of
the conditions in (\ref{eq:frobenius}) has measure zero in $\mathbb{R}$.
(This is due to Sard's Theorem \citep{guillemin1974differential}.)
Because of inevitable numerical inaccuracies and imprecisions, with
probability 1, the collection of practically available $\xi_{1,2,3,4}$-fields
will hence produce a regular value for each of the Frobenius conditions
in (\ref{eq:frobenius}). This brings us back to \emph{Case 1}. 

We conclude that only \emph{Case 1} is relevant in practice. (Unless,
of course, a special symmetry of the flow map $F_{t_{0}}^{t_{1}}$
implies that the Frobenius conditions (\ref{eq:frobenius}) are not
independent to begin with.) Straightforwardly extending Definition
\ref{def:hyperbolicLCS} and, therefore, seeking hyperbolic repelling
LCSs as surfaces normal to $\xi_{4}$ is not a useful approach for
general dynamical systems $\dot{x}=u(x,t)$ in four dimensions. 
\end{example*}

The above discussion holds in any dimension $N\in\left\{ 4,5,...\right\} $
and for any LCS type: From a collection of $N-1$ vector fields, we
can pick $f=\binom{N-1}{2}$ pairs, yielding precisely $f$ Frobenius
conditions (cf. (\ref{eq:frobenius})). For useful and general LCS
definitions in the spirit of Sec. \ref{sec:Review3d}, we would generally
need $f=1$, but this is only achieved for $N=3$. This precludes
straightforward extensions of Theorem \ref{thm:xi2_dual} from three
to higher dimensions.

\section{Numerical details for the examples\label{sec:details-examples}}

Here we describe the details of our numerical approach. These apply
to all the examples in Sec. \ref{sec:Examples}.

In order to evaluate $\xi_{2}$, we need to compute both the flow
map $F_{t_{0}}^{t_{1}}$ and its derivative $DF_{t_{0}}^{t_{1}}$.
Here we do not use finite differentiating in order to obtain $DF_{t_{0}}^{t_{1}}$
from $F_{t_{0}}^{t_{1}}$(cf., e.g., \citep{Haller2015}), but we
explicitly solve for $DF_{t_{0}}^{t_{1}}$. Since the flow map $F_{t_{0}}^{t}$
satisfies 
\begin{equation}
\frac{d}{dt}F_{t_{0}}^{t}(x_{0})=u(F_{t_{0}}^{t}(x_{0}),t),\label{eq:flowmap-ode}
\end{equation}
we differentiate (\ref{eq:flowmap-ode}) with respect to $x_{0}$,
and conclude that $DF_{t_{0}}^{t}(x_{0})$ evolves according to the
well-known equation of variations
\begin{equation}
\frac{d}{dt}DF_{t_{0}}^{t}(x_{0})=Du\left(F_{t_{0}}^{t}(x_{0}),t\right)DF_{t_{0}}^{t}(x_{0}).\label{eq:eq-of-vari}
\end{equation}
Written out in coordinates, (\ref{eq:eq-of-vari}) is a system of
nine equations that is coupled to the three equations in (\ref{eq:flowmap-ode})
and, therefore, both (\ref{eq:flowmap-ode}) and (\ref{eq:eq-of-vari})
need to be solved simultaneously as a system of 12 variables. We can
thus obtain $DF_{t_{0}}^{t_{1}}$ and $\xi_{2}$ to very high accuracy,
which we need for running long integral curves of (\ref{eq:xi2_ODE}). 

Once $DF_{t_{0}}^{t_{1}}$ is available, rather than using the Cauchy-Green
strain tensor \citep{Haller2015}, we obtain $\xi_{2}$ by SVD (cf.
Remark \ref{remark:Cauchy-Green} and \citep{Watkins2005NoEIGButSVD}).
(For $\eta_{2}$, we use the backward-time deformation gradient $DF_{t_{1}}^{t_{0}}$.)

We do not compute the $\xi_{2}$-field on a spatial grid, but just
along the $\xi_{2}$-lines that we integrate. This ensures that we
can locate both small and highly-modulated LCSs, instead of risking
to accidentally undersample unknown structures. At each point of the
curve, we assign the orientation of $\xi_{2}$ to be the same as it
was at the previous point on the curve. For the initial point, one
has to make a manual choice; e.g., in Cartesian coordinates $(x,y,z)$,
impose alignment with the $(0,0,1$)-direction. 

We perform all the integrations using a Runge-Kutta (4,5) method \citep{Dormand1980},
with an adaptive stepper at absolute and relative error tolerances
of $Tol=10^{-8}$. 

Finally, we obtain all the Poincaré maps from trajectories (of either
$u$, $\xi_{2}$, or $\eta_{2}$) by plotting the $(x,y)$-point data
corresponding to $z$-values from $[0,\text{\ensuremath{\epsilon}}]\cup[2\pi-\epsilon,2\pi]$,
with $\epsilon=2\cdot10^{-3}$. 

For the steady ABC flow (cf. Sec. \ref{sub:ABCsteady}), we evaluate
how the equation of variations (\ref{eq:eq-of-vari}) improves the
results for $\xi_{2}$ compared to finite differencing of $F_{t_{0}}^{t_{1}}$
(cf. \citep{Haller2015}). We define a uniform rectangular grid of
500$\times$500 initial conditions $x_{0}$ in the plane given by
$\mbox{\{\ensuremath{(x,y,0):}\,\ensuremath{x,y\in[0,2\pi]}\}}$,
for which we evaluate $DF_{t_{0}}^{t_{1}}$ and thus $\xi_{2}$ using
these two methods. We perform finite differencing as described in
\citep{Haller2015}, Eq. 9, with $\delta_{1,2,3}=10^{-5}e_{1,2,3}$
and $e_{1,2,3}$ denoting the unit vectors in the $x,y,z$ coordinate
directions. 
\begin{figure}[h]
\centering{}\includegraphics[width=0.8\columnwidth]{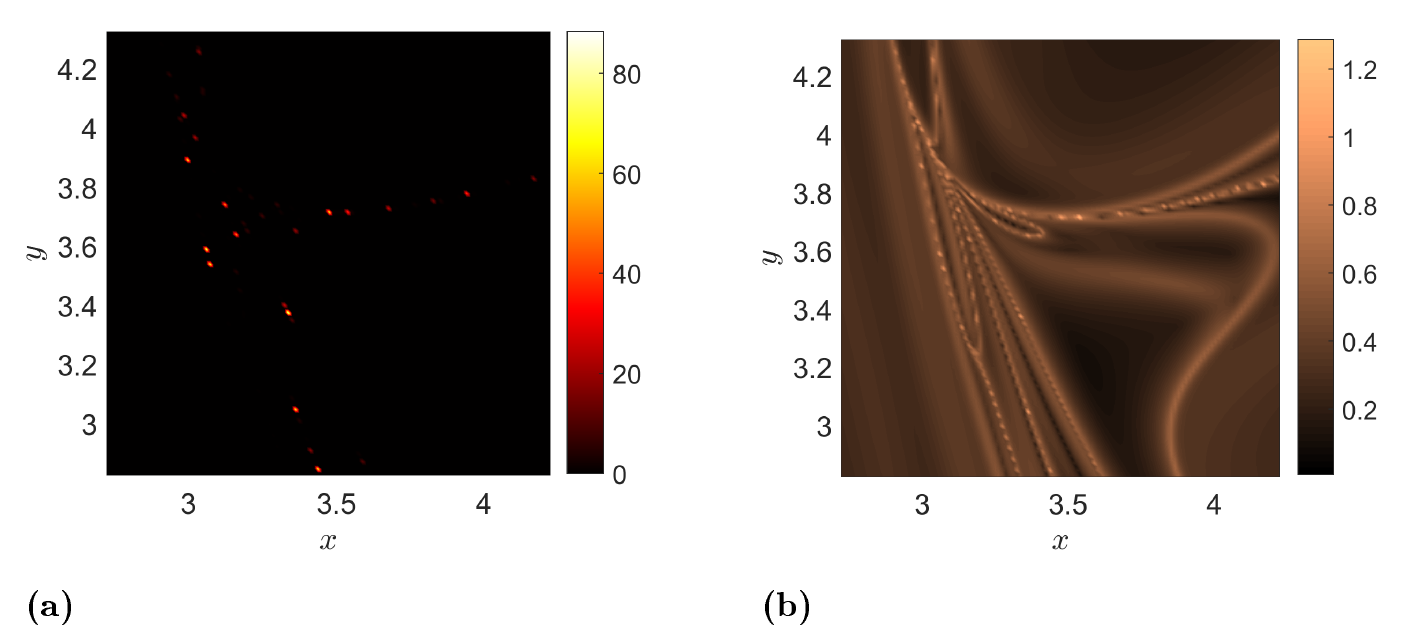}\subfloat{\label{fig:ABCsteady_xi2Error}}\subfloat{\label{fig:ABCsteady_FTLE}}\protect\caption{Steady ABC flow: Error due to finite differencing. (a) Angle in degrees
between $\xi_{2}$ obtained from finite differencing of $F_{t_{0}}^{t_{1}}$
(cf. \citep{Haller2015}) and $\xi_{2}$ obtained using the equation
of variations (\ref{eq:eq-of-vari}). (b) FTLE $(t_{1}-t_{0})^{-1}\log\sigma_{3}$
obtained using the equation of variations (\ref{eq:eq-of-vari}).}
\end{figure}
 In Fig. \ref{fig:ABCsteady_xi2Error}, we show the angle between
$\xi_{2}$ obtained using (\ref{eq:eq-of-vari}) and $\xi_{2}$ obtained
from finite differencing of $F_{t_{0}}^{t_{1}}$. The former method
can be considered practically exact here, with the only numerical
parameter being $Tol=10^{-8}$ (checked for convergence). The largest
error we find in Fig. \ref{fig:ABCsteady_xi2Error} is approximately
$88.35{^\circ}$. Since $\xi_{2}$ is only defined up to orientation,
the largest possible error would be $90{^\circ}$. Hence we conclude
from Fig. \ref{fig:ABCsteady_xi2Error} that finite differencing can
cause arbitrarily large errors in $\xi_{2}$. Even though errors are
confined to locations of exceptionally large separation, as indicated
by the finite-time Lyapunov exponent (FTLE) field (cf. Fig. \ref{fig:ABCsteady_FTLE}),
these locations belong to ridges of the FTLE field, a widely used
indicator of hyperbolic LCSs \citep{Haller2015}. Since we want to
globally detect hyperbolic LCSs by integrating the $\xi_{2}$-field,
we use (\ref{eq:eq-of-vari}) to determine $\xi_{2}$.

We note that even when the velocity field (\ref{eq:flowdef}) is only
available through data from experiments and simulations, the equation
of variations (\ref{eq:eq-of-vari}) has been used to obtain numerically
accurate results for the flow map and its gradient \citep{Miron20126419}.

\section{Perturbations to the $\xi_{2}$-field\label{sec:perturbations}}

In Figs. \ref{fig:ABCaperiodic_repLCS_t0p0}, \ref{fig:ABCaperiodic_repLCS_t1p0},
we place a tracer sphere in an LCS candidate surface, finding that
it stretches most in the direction normal to the surface. Based on
this local property, in Sec. \ref{sub:ABCaperiodic5}, we conclude
that the entire surface should be a repelling LCS. Even though we
expect any hyperbolic LCS obtained from a forward-time computation
to be repelling (cf. Remark \ref{rem:forward-repelling}), it is desirable
to have a global approach to assessing the LCS type of a candidate
surface.

If we consider, e.g., a repelling LCS $\mathcal{M}(t_{0})$, at any
point $x_{0}\in\mathcal{M}(t_{0})$, the tangent space $T_{x_{0}}\mathcal{M}(t_{0})$
is the subspace of $\mathbb{R}^{3}$ spanned by $\xi_{2}(x_{0})$
and $\xi_{1}(x_{0})$ (cf. Proposition \ref{prop:repellingLCS}).
By repeating the reasoning that leads to Theorem \ref{thm:xi2_dual},
we conclude that any repelling LCS $\mathcal{M}(t_{0})$ must be an
invariant manifold of any dynamical system of the form
\[
x_{0}'=p\,\xi_{2}(x_{0})+(1-p)\xi_{1}(x_{0}),\quad p\in[0,1].
\]

By Propositions \ref{prop:attractingLCS}--\ref{prop:ellipticLCS},
we can make similar observations for the remaining LCS types. In summary: 
\begin{prop}
\label{prop:dual_family}For any parameter value \foreignlanguage{english}{\textup{$p\in[0,1]$}},
the initial position $\mathcal{M}(t_{0})$ of any hyperbolic or elliptic
LCS (Definitions \ref{def:hyperbolicLCS}--\ref{def:ellipticLCS})
is an invariant manifold of the autonomous dual dynamical system 
\begin{equation}
x_{0}^{\prime}=p\,\xi_{2}(x_{0})+(1-p)\tilde{\xi}(x_{0})\,,\label{eq:xip_ODE}
\end{equation}
with $\mbox{\ensuremath{\tilde{\xi}}=\ensuremath{\xi}}_{3}$ for attracting
hyperbolic LCSs; $\tilde{\xi}=\xi_{1}$ for repelling hyperbolic LCSs;
and $\tilde{\xi}=\mp\tilde{\gamma}\xi_{1}+\tilde{\alpha}\xi_{3}$
or $\tilde{\xi}=\mp\gamma\xi_{1}+\alpha\xi_{3}$ for elliptic LCSs
(cf. (\ref{eq:shear_normal}), (\ref{eq:elliptic_normal})).\end{prop}
\begin{rem}
Replacing the $\xi_{1,2,3}$ by $\sigma_{1,2,3}\cdot\eta_{1,2,3}$,
Proposition \ref{prop:dual_family} applies\emph{ verbatim} to final
LCS positions $\mathcal{M}(t_{1})$.
\end{rem}

This means that for each LCS type, there is a specific family of dual
dynamical systems that yields the respective LCS initial positions
as invariant manifolds. The dual dynamical system associated with
$\xi_{2}$ remains exceptional though, because this is the only dual
dynamical system shared by all LCS types (cf. Proposition \ref{prop:dual_family}). 

We now demonstrate how these observations help to determine the LCS
type of a candidate surface: For the hyperbolic LCS candidate in the
time-aperiodic ABC-type flow (cf. Sec. \ref{sub:ABCaperiodic5}),
it turns out that only a single long $\xi_{2}$-line is enough to
indicate the surface (cf. Fig. \ref{fig:ABCaperiodic_xi2-line-hyperbolic}).
Specifically, among the $\xi_{2}$-lines that get attracted to the
hyperbolic LCS candidate surface in the dual Poincaré map (cf. Fig.
\ref{fig:ABCaperiodic_xi2Pmap}), we have randomly picked the $\xi_{2}$-line
with initial condition approximately equal to $(5.03,3.14,0.00)$.
Other choices of $\xi_{2}$-lines yield similar results. 
\begin{figure}[h]
\includegraphics[width=0.95\linewidth]{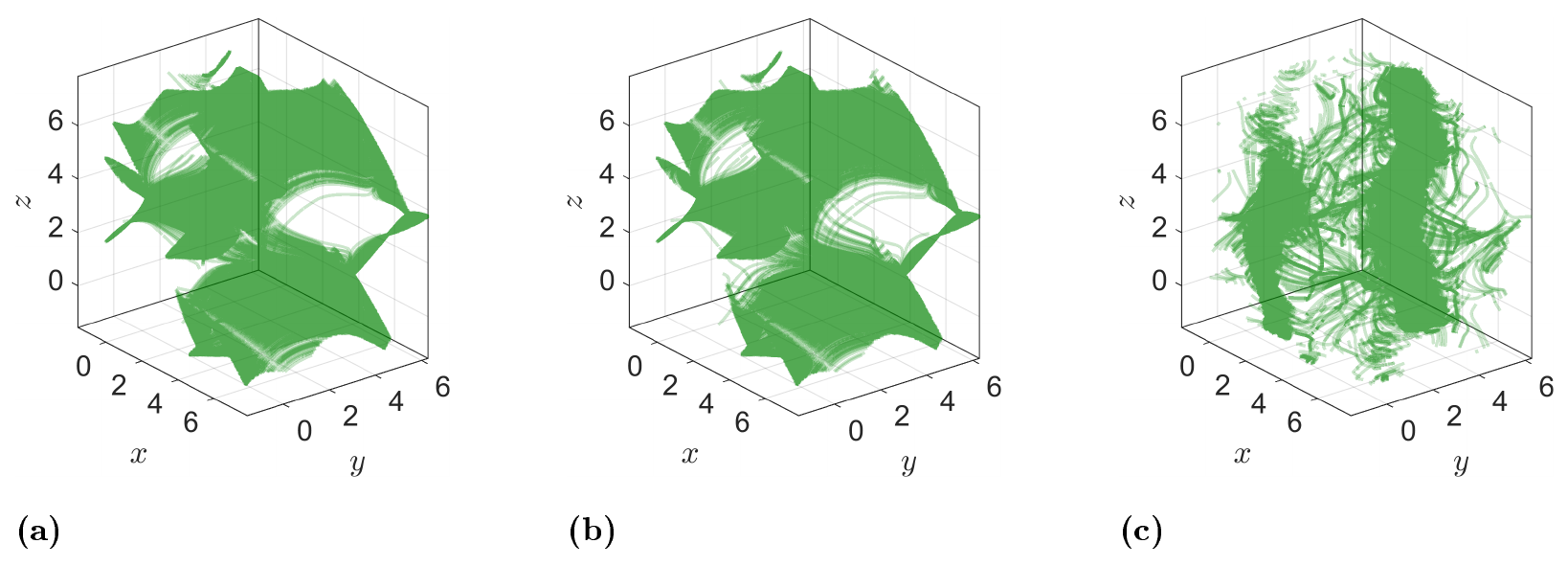}\subfloat{\label{fig:ABCaperiodic_xi2-line-hyperbolic}}\hfill{}\subfloat{\label{fig:ABCaperiodic_xi2xi1-line-hyperbolic}}\hfill{}\subfloat{\label{fig:ABCaperiodic_xi2xi3-line-hyperbolic}}\protect\caption{Time-aperiodic ABC-type flow: Arc segments of integral curves of three
$\xi_{2}+\epsilon\tilde{\xi}$ fields. (Each curve is shown for arclength
parameter $s\in[4\cdot10^{4},5\cdot10^{4}]$). The initial condition
is approximately $(5.03,3.14,0.00)$ for all three integral curves.
Here we use the periodicity of the phase space to extend the domain
slightly beyond $[0,2\pi]^{3}$. (a) A $\xi_{2}$-line ($\epsilon=0$)
indicates the hyperbolic candidate surface obtained from the dual
Poincaré map (cf. Fig. \ref{fig:ABCaperiodic_xi2Pmap}). (b) An integral
curve of $\xi_{2}+\epsilon\xi_{1}$ ($\epsilon=0.01$) reproduces
the hyperbolic candidate surface obtained from the corresponding $\xi_{2}$-line
(cf. Fig. \ref{fig:ABCaperiodic_xi2-line-hyperbolic}). (c) An integral
curve of $\xi_{2}+\epsilon\xi_{3}$ ($\epsilon=0.01$) does not reproduce
the hyperbolic candidate surface obtained from the corresponding $\xi_{2}$-line
(cf. Fig. \ref{fig:ABCaperiodic_xi2-line-hyperbolic}).}
\end{figure}

We next add a small perturbation to the $\xi_{2}$-field, i.e., consider
the dual dynamical system
\begin{equation}
x_{0}^{\prime}=\xi_{2}(x_{0})+\epsilon\xi_{1}(x_{0}),\label{eq:xi2_tang_pert}
\end{equation}
with $\epsilon=0.01$. Using the same initial condition and numerical
settings as above, we compute an integral curve of (\ref{eq:xi2_tang_pert}).
The result indicates virtually the same surface as obtained from the
$\xi_{2}$-field (cf. Fig. \ref{fig:ABCaperiodic_xi2xi1-line-hyperbolic}).
This suggests that this surface is invariant for the entire family
of direction fields $p\xi_{2}+(1-p)\xi_{1}$. By Proposition \ref{prop:dual_family},
the entire structure should hence be a repelling LCS.

If we, on the other hand, repeat the above computation for the dual
dynamical system
\begin{equation}
x_{0}^{\prime}=\xi_{2}(x_{0})+\epsilon\xi_{3}(x_{0}),\label{eq:xi2_norm_pert}
\end{equation}
where $\epsilon=0.01$, then the entire structure disappears, and
the attractor for this initial condition remains unclear (cf. Fig.
\ref{fig:ABCaperiodic_xi2xi3-line-hyperbolic}).  Even though the
perturbation $\epsilon\xi_{3}$ is small, the dynamics of (\ref{eq:xi2_norm_pert})
is completely different than for (\ref{eq:xi2_tang_pert}). This is
consistent with our conclusion that the structure from Figs. \ref{fig:ABCaperiodic_xi2-line-hyperbolic},
\ref{fig:ABCaperiodic_xi2xi1-line-hyperbolic} is a repelling hyperbolic
LCS. 

\bibliographystyle{abbrv}

\end{document}